\def\version{15 September 2013}

\documentclass[a4paper,11pt]{amsart}
\usepackage[margin=28mm]{geometry}
\usepackage{amsfonts,amsmath,amssymb,amsthm}
\usepackage{fixmath}
\usepackage{paralist}
\usepackage{tikz}
\usepackage[backref=page]{hyperref}

\usepackage{lineno}
\usepackage[math]{lnenv}
\linenumbers

\hypersetup{%
pdftitle={Linear and projective boundary of nilpotent groups},
pdfsubject={Mathematics},
pdfauthor={Bernhard Kr\"on, J\"org Lehnert, Norbert Seifter, Elmar Teufl},
pdfkeywords={metric spaces, boundaries, nilpotent groups},
pdfborder={0 0 0},colorlinks=true}

\theoremstyle{plain}
\newtheorem{theorem}{Theorem}[section]
\newtheorem{lemma}[theorem]{Lemma}
\newtheorem{corollary}[theorem]{Corollary}
\newtheorem{proposition}[theorem]{Proposition}
\theoremstyle{definition}
\newtheorem{definition}[theorem]{Definition}
\newtheorem{example}[theorem]{Example}
\theoremstyle{remark}
\newtheorem*{remark}{Remark}

\def\N{\mathbb N}
\def\Z{\mathbb Z}
\def\R{\mathbb R}
\def\S{\mathbb S}
\def\P{\mathbb P}

\def\Ball{B}
\def\OpenBall{U}
\def\Graph{X}
\def\Unbdd{\mathcal U}
\def\Elements{\mathcal E}

\def\LieG{\mathfrak g}

\def\Orbit{\mathcal C}
\def\Forward{\mathcal C^+}
\def\Project{\mathcal P}
\def\Linear{\mathcal L}

\def\linequiv{\sim}
\def\Quotient#1{#1/\mathord\linequiv}
\def\interior{\qopname\relax o{\mathsf{int}}}
\def\closure{\qopname\relax o{\mathsf{cl}}}
\def\Bndry#1{\closure(\Quotient{#1})}

\def\rank{\qopname\relax o{\mathsf{rk}}}
\def\vdim{\qopname\relax o{\mathsf{dim}}}
\def\vspan{\qopname\relax o{\mathsf{span}}}

\def\CAT{\qopname\relax o{CAT}}
\def\Aut{\qopname\relax o{\mathsf{Aut}}}
\def\UnbddOrb{\qopname\relax o{\mathsf{Unb}}}

\def\phi{\varphi}
\def\eps{\varepsilon}
\def\setsep{\,:\,}
\def\ndash{\nobreakdash-\hskip0pt}
\def\rightfactor{\backslash}
\def\argument{\mathord{\cdot}}

\def\gen#1{\langle#1\rangle}
\def\genp#1{\langle#1\rangle^+}
\def\abs#1{\lvert#1\rvert}

\def\norm#1{\lVert#1\rVert}
\def\lgauge{\mathopen{\kern-3pt\left\bracevert\vphantom{f}\right.\kern-4pt}}
\def\rgauge{\mathclose{\kern-4pt\left.\vphantom{f}\right\bracevert\kern-3pt}}
\def\gauge#1{\lgauge#1\rgauge}

\makeatletter
\def\section{\def\@secnumfont{\bfseries}%
  \@startsection{section}{1}%
  \z@{.7\linespacing\@plus\linespacing}{.5\linespacing}%
  {\normalfont\bfseries\centering}}
\makeatother

\begin{document}

\title[Linear and projective boundary]{Linear and projective boundary of nilpotent groups}

\author{Bernhard Kr\"on}
\address{Bernhard Kr\"on\\
	Fakult\"at f\"ur Mathematik\\
	Universit\"at Wien\\
	Nordbergstra\ss e~15\\
	1090 Vienna\\
	Austria}
\email{bernhard.kroen@univie.ac.at}

\author{J\"org Lehnert}
\address{J\"org Lehnert\\
	Max Planck Institute for Mathematics in the Sciences\\
	Inselstra\ss e~22\\
	04103 Leipzig\\
	Germany}
\email{lehnert@mis.mpg.de}

\author{Norbert Seifter}
\address{Norbert Seifter\\
	Department Mathematik und Informationstechnologie\\
	Montanuniversit\"at Leoben\\
	Franz-Josef-Stra\ss e~18\\
	8700 Leoben\\
	Austria}
\email{seifter@unileoben.ac.at}

\author{Elmar Teufl}
\address{Elmar Teufl\\
	Mathematisches Institut\\
	Universit\"at T\"ubingen\\
	Auf der Morgenstelle~10\\
	72076 T\"ubingen\\
	Germany}
\email{elmar.teufl@uni-tuebingen.de}

\date{\version}


\begin{abstract}
We define a pseudometric on the set of all unbounded subsets of a metric space. The Kolmogorov quotient of this pseudometric space is a complete metric space. The definition of the pseudometric is guided by the principle that two unbounded subsets have distance $0$ whenever they stay sublinearly close. Based on this pseudometric we introduce and study a general concept of boundaries of metric spaces. Such a boundary is the closure of a subset in the Kolmogorov quotient determined by an arbitrarily chosen family of unbounded subsets.

Our interest lies in those boundaries which we get by choosing unbounded cyclic sub\-\hbox{(semi)}\-groups of a finitely generated group (or more general of a compactly generated, locally compact Hausdorff group). We show that these boundaries are quasi-isometric invariants and determine them in the case of nilpotent groups as a disjoint union of certain spheres (or projective spaces).

In addition we apply this concept to vertex-transitive graphs with polynomial growth and to random walks on nilpotent groups.
\end{abstract}

\maketitle

\tableofcontents

\section{Introduction}
\label{section:introduction}

There are numerous boundary notions of graphs, groups, manifolds, metric spaces and other geometric objects. The literature on the subject is extensive and boundary notions proved to be a useful tool in studying the underlying space. An early instance is the theory of ends which was developed in the first half of the twentieth century by Freudenthal (see e.g.~\cite{freudenthal1942neuaufbau}) and others. Various geometric ideas were used to refine the notion of ends:

In 1973 Eberlein and O'Neill~\cite{eberlein1973visibility} constructed the boundary at infinity of a $\CAT(0)$~space by considering equivalence classes of non-compact geodesic rays. The equivalence notion of geodesic rays uses the natural parametrization, i.e. two geodesic rays are equivalent if they stay at bounded distance as the parameter tends to $\infty$. A different description is given by Gromov in \cite{gromov1981hyperbolic} which uses an embedding into the set of continuous functions relying on the metric only.

In graph theory in the 1990s Jung~\cite{jung1993notes} and Jung, Niemayer~\cite{jung1995decomposing} introduced a refinement of ends of graphs called b\ndash fibers and d\ndash fibers, respectively. The basic idea behind fibers is to consider points at infinity as equivalence classes of rays (infinite paths) which stay at bounded distance ``up to linear reparametrization''. In 2005 Bonnington, Richter and Watkins~\cite{bonnington2007between} modified this concept by considering rays as equivalent whenever they stay at sublinear distance ``up to linear reparametrization''. They were able to use this concept to prove some nice results on infinite planar graphs, but the boundary, whose elements have been called ``bundles'', was not topologized and not considered for groups or vertex-transitive graphs.

Another instance, where the concept of staying at sublinear distance is used, is given by Kaimanovich in \cite[Theorem~5.5]{kaimanovich1991poisson}. The so-called ``ray approximation'' is used to determine, whether a given probability space is the Poisson boundary of a random walk on a countable group $G$ defined by a probability measure $\mu$ on $G$. A proposal space $(B,\lambda)$ is the Poisson boundary of $(G,\mu)$, if compatibility conditions between $\mu$ and $\lambda$ hold and if there exist measureable ``projections'' $\pi_n\colon B\to G$, such that almost every trajectory $(g_1,g_2,\dotsc)$ stays sublinear close to $(\pi_1(g_\infty),\pi_2(g_\infty),\dotsc)$, where $g_\infty$ is the limit point of the trajectory $(g_1,g_2,\dotsc)$ in $B$.

In these examples the ``parametrization'' of rays or sequences is used in the definition of staying (sublinearly) close. In the following we relax this and work with general subsets and not only with rays or sequences. Let $(X,d)$ be a metric space, let $o\in X$ be a fixed reference point and denote by $B(x,r)$ the closed ball in $(X,d)$ with center $x$ and radius $r$. If $R,S$ are two unbounded subsets of $X$, their distance $t(R,S)$ is defined to be the square root of the infimum over all $\alpha\ge0$, such that
\[ S\subseteq\bigcup_{x\in R} B(x, \alpha d(o,x)+a) \qquad\text{and}\qquad
	R\subseteq\bigcup_{y\in S} B(y, \alpha d(o,y)+a) \]
for some $a\ge0$. The sets $R,S$ are \emph{sublinearly close}, if $t(R,S)=0$. We show that the set of all unbounded subsets of $(X,d)$ equipped with the distance $t$ is a pseudometric space, whose Kolmogorov quotient is a complete metric space (Proposition~\ref{proposition:premetric} and Theorem~\ref{theorem:metric}). Given some family $\Elements$ of unbounded subsets the associated ``boundary'' of $X$ is the closure of all equivalence classes which contain an element from $\Elements$ in the Kolmogorov quotient. Interesting families of unbounded subsets include: geodesics, horoballs, cyclic sub\-\hbox{(semi)}\-groups (in the case of groups), one-parameter sub\-\hbox{(semi)}\-groups (in the case of topological groups).

We mainly focus on the group case. Let $G$ be a finitely generated (or more general compactly generated, locally compact Hausdorff) group and let $d$ be a word metric on $G$. If the family $\Elements$ is given by all unbounded cyclic subsemigroups or by all unbounded cyclic subgroups, we call the associated boundary \emph{linear boundary} in the former case and \emph{projective boundary} in latter case. We prove that these two boundaries are quasi-isometric invariants (Lemma~\ref{lemma:boundary}). In our main result (Theorem~\ref{theorem:nilpotent}) we identify the linear and projective boundary for nilpotent groups. Let $G$ be either a connected, nilpotent Lie group or a finitely generated, nilpotent group with descending central series
\[ G = \gamma_1(G) \supseteq \gamma_2(G) \supseteq \dotsb \supseteq \gamma_c(G) \supsetneq \gamma_{c+1}(G) = 1. \]
Let $\nu(i)$ denote the compact-free dimension or the torsion-free rank of the commutative group $\gamma_i(G)/\gamma_{i+1}(G)$. Then the linear boundary is homeomorphic to the disjoint union of $c$ spheres
\[ \S^{\nu(1)-1} \uplus \S^{\nu(2)-1}\uplus \dotsb \uplus \S^{\nu(c)-1} \]
and the projective boundary is homeomorphic to the disjoint union of projective spaces
\[ \P^{\nu(1)-1} \uplus \P^{\nu(2)-1}\uplus \dotsb \uplus \P^{\nu(c)-1}. \]
Here $\S^d$ is the $d$\ndash dimensional sphere and $\P^d$ is the $d$\ndash dimensional projective space.

The following facts about the boundary notion introduced above must be stressed: Compact elements of a group $G$ do not contribute to the boundaries $\Project G$ and $\Linear G$. Hence, whenever $G$ only contains compact elements, these boundaries are empty. In particular, this means in the discrete case, that $\Project G$ and $\Linear G$ are empty for torsion groups. We also emphasize that we compare unbounded sets using the distance function $t$ and not sequences or rays using their parametrizations, as it is e.g. done in \cite{bonnington2007between}. For instance, two sequences or rays might be distant in the sense of \cite{bonnington2007between} using parametrizations $a_n=n$ and $b_n=\sqrt n$, respectively, although the underlying unbounded sets are the same hence sublinearly close.

\vspace{2mm}

The paper is organized as follows:
\begin{compactitem}
\item The general framework for metric spaces is studied in Section~\ref{section:construction}.
	The distance $t$ and some auxiliary quantities are introduced and
	several basic results are proved.
	For instance we show that the distance $t$ has all properties stated above.
\item In Section~\ref{section:quasi-iso} we investigate the relationship to quasi-isometries.
	It is shown that the distance $t$ is preserved up to bi-Lipschitz-equivalence under
	quasi-isometries of the underlying space (Theorem~\ref{theorem:qi-bilip}).
\item In Section~\ref{section:cat} we show that
	the boundary at infinity of a complete $\CAT(0)$ space
	equipped with the angular metric can be obtained by
	the boundary construction outlined above using
	the set of unbounded geodesics up to bi-H\"older equivalence.
\item In Section~\ref{section:groups} we apply this concept to groups using
	unbounded cyclic sub\-\hbox{(semi)}\-groups as families of unbounded subsets.
	Some general results are obtained and the case of abelian groups is discussed in detail.
	In the latter case the projective boundary is homeomorphic to a projective space
	and the linear boundary is homeomorphic to a sphere.
\item Section~\ref{section:nilpotent} is devoted to the formulation and proof
	of the main result (Theorem~\ref{theorem:nilpotent}).
	Most technical parts of the proof are deferred to Appendix~\ref{appendix:groups}.
\item Section~\ref{section:graph-boundaries} discusses the situation for graphs.
	The projective boundary of a graph is defined by the above procedure,
	using the family of unbounded orbits generated by cyclic subgroups of the automorphism group
	and the linear boundary is defined analogously.
	For connected, vertex-transitive graphs with polynomial volume growth
	we obtain the same description of the projective (respectively linear) boundary
	as in the case of nilpotent groups (Corollary~\ref{corollary:grdis}).
\item In Section~\ref{section:attach} we construct a topology on the disjoint union of the base space $X$
	and some boundary which is obtained by the construction above.
	The definition is reminiscent of the cone topology of the boundary at infinity of $\CAT(0)$ spaces.
	The subspace topology on $X$ of this topology is always induced by the metric $d$,
	but the subspace topology on the boundary is neither induced by $t$ nor Hausdorff in general.
	We discuss criteria (Lemma~\ref{lemma:hausdorff} and Proposition~\ref{proposition:group-bnd})
	which guarantee both: the subspace topology on the boundary is Hausdorff and induced by $t$.
\item In Section~\ref{section:walks} we show that every boundary point in the linear boundary
	of a nilpotent Lie group is obtained as a limit of a random walk with drift and vice versa.
\item Appendix~\ref{appendix:compact} collects some known results on compactly generated groups and word metrics
	which are used in the previous sections.
\item Appendix~\ref{appendix:groups} mostly contains the technical lemmas
	used in the proof of Theorem~\ref{theorem:nilpotent} and
	the necessary notions from Lie theory.
\end{compactitem}

\section{General construction}
\label{section:construction}

Let $(X,d)$ be a metric space. We write $\Unbdd$ to denote the family of unbounded subsets of $(X,d)$. The closed and open ball with center $x\in X$ and radius $r\ge0$ in $(X,d)$ are denoted by
\[ \Ball(x,r) = \{y \in X \setsep d(y,x)\le r\} \qquad\mbox{and}\qquad
	\OpenBall(x,r) = \{y \in X \setsep d(y,x)< r\}, \]
respectively. Let $o$ be a fixed reference point, let $R\subseteq X$, and let $\alpha$ and $a$ be nonnegative real numbers. We set
\[ \alpha R+a = \bigcup_{x\in R} \Ball(x,\alpha d(o,x)+a) \]
and write $\alpha R$ instead of $\alpha R+0$.

\begin{remark}
The notation $\alpha R+a$ is unusual, but turns out to be convenient for computations involving sets of this form. Mostly this notation will be used if $X$ is a metric space, so no confusion should occur. However, if $X$ is a linear space too, $\alpha R+a$ will always be used in the above meaning and never means a linearly scaled and translated set. Furthermore, it should be the stressed that $0R = R$ and
\[ 0R+a = \bigcup_{x\in R} \Ball(x,a), \]
which is often called $a$\ndash neighborhood of $R$ or generalized ball of radius $a$ around $R$.
\end{remark}

\begin{lemma}\label{lemma:greater-one}
Let $R\in\Unbdd$ and $\alpha>1$. Then $\alpha R=X$.
\end{lemma}

\begin{proof}
Let $x$ be any point in $X$. Since $R$ is unbounded, there is an element $y\in R$ such that $(\alpha-1)d(o,y)\ge d(o,x)$. Hence
\[ d(x,y) \le d(x,o)+d(o,y) \le (\alpha-1)d(o,y)+d(o,y) = \alpha d(o,y),\]
and $x\in\Ball(y,\alpha d(o,y)) \subseteq \alpha R$.
\end{proof}

\begin{lemma}\label{lemma:transitive}
Let $R$, $S$, and $T$ be subsets of $X$. If $T \subseteq \beta S+b$ and $S \subseteq \alpha R+a$ then $T \subseteq (\alpha+\alpha\beta+\beta)R+\beta a+a+b$.
\end{lemma}

\begin{proof}
Let $z$ be in $T$. Since the sets $\alpha R+a$ and $\beta S+b$ are defined as unions of balls, $z$ is in $\Ball(y,\beta d(o,y)+b)$ for some $y\in S$ and $y$ is in $\Ball(x,\alpha d(o,x)+a)$ for some $x\in R$.  Set $c=d(o,x)$. Then $d(x,y)\le  \alpha c+a$. By the triangle inequality,
\[ d(o,y)\le d(o,x)+d(x,y)\le (\alpha +1)c+a. \]
Hence
\[ d(y,z)\le \beta d(o,y)+b\le (\alpha \beta+\beta)c+\beta a+b. \]
Finally,
\[ d(x,z)\le d(x,y)+d(y,z)\le (\alpha +\alpha \beta+\beta)c+\beta a+a+b. \]
This means that $z\in(\alpha +\alpha \beta+\beta)R+\beta a+a+b$.
\end{proof}

\begin{definition}\label{definition:lequiv}
For two subsets $R,S\subseteq X$ let $s^+(R,S)$ be the infimum of all $\alpha\ge0$ such that $S\subseteq \alpha R+a$ for some $a\ge0$. Set
\[ s(R,S) = \max\{s^+(R,S),s^+(S,R)\} \]
and $t(R,S)=\sqrt{s(R,S)}$. If $R,S\in \Unbdd$ and $s(R,S)=0$ then $R$ and $S$ are called \emph{linearly equivalent} and we write $R\linequiv S$.
\end{definition}

\begin{remark}
The functions $s^+$, $s$, $t$ depend on the metric space $(X,d)$. In order to emphasize the underlying metric space $(X,d)$ we write $s^+_X$ or $s^+_{\smash{(X,d)}}$ and analogously for $s$ and $t$. Similarly, we write $\Unbdd_X$ or $\Unbdd_{\smash{(X,d)}}$ instead of $\Unbdd$, if it is necessary to specify the metric space.
\end{remark}

\begin{lemma}\label{lemma:reference}
Let $R,S$ be two subsets of $X$. Then $s^+(R,S)$ and therefore $s(R,S)$ and $t(R,S)$ do not depend on the reference point $o$ in $X$.
\end{lemma}

\begin{proof}
Let $o,p\in X$ and set $c=d(o,p)$. We write $s^+_o(R,S)$ in order to emphasize the reference point $o$. Furthermore, write $C_o(R,\alpha,a)$ to denote the set $\alpha R+a$ with respect to the reference point $o$. For $\alpha>s_o^+(R,S)$ there is a number $a>0$ such that $S\subseteq C_o(R,\alpha,a)$. Hence for $y\in S$ we can find a point $x\in R$ such that $d(y,x) \le \alpha d(o,x) + a$. The triangle inequality implies that
\[ d(y,x) \le \alpha ( d(p,x) + d(p,o) ) + a = \alpha d(p,x) + \alpha c + a. \]
Therefore $S\subseteq C_p(R,\alpha,\alpha c+a)$ and thus $s_p^+(R,S)\le s_o^+(R,S)$. The reversed inequality is obtained by changing the r\^ole of $o$ and $p$.
\end{proof}

\begin{lemma}\label{lemma:alternative}
Let $R,S\in\Unbdd$. Then $s^+(R,S)$ is the infimum of all $\alpha\ge0$ such that $S\setminus\OpenBall(o,r) \subseteq \alpha R$ for some $r\ge0$.
\end{lemma}

\begin{proof}
We write $\sigma^+(R,S)$ to denote the infimum of all $\alpha\ge0$ such that $S\setminus\OpenBall(o,r) \subseteq \alpha R$ for some $r\ge0$. First we show that $s^+(R,S) \le \sigma^+(R,S)$. Assume that $\alpha>\sigma^+(R,S)$ and $r\ge0$ such that $S\setminus\OpenBall(o,r) \subseteq \alpha R$. Set $a = r + d(o,R)$, where $d(o,R) = \inf\{d(o,x) \setsep x\in R\}$. Then, by triangle inequality, $S\subseteq \alpha R+a$. Therefore $s^+(R,S)\le \alpha$ and hence $s^+(R,S)\le \sigma^+(R,S)$.

Now we prove the reversed inequality: Let $\alpha>s^+(R,S)$ and set $\eps = \tfrac12(\alpha-s^+(R,S)) > 0$. Then, by definition of $s^+(R,S)$, there exists a constant $a\ge0$ such that $S\subseteq (\alpha-\eps)R+a$. We claim that $S\setminus\OpenBall(o,r) \subseteq \alpha R$ holds for $r = \tfrac a \eps(1+\alpha)$. Let $y\in S$. Then there is a $x\in R$ with $d(y,x)\le(\alpha-\eps)d(o,x)+a$. Using the triangle inequality yields
\[ d(o,y) \le d(o,x) + d(y,x) \le d(o,x) + (\alpha-\eps) d(o,x)+a \le (1+\alpha-\eps) d(o,x) + a, \]
which implies
\[ d(o,x) \ge \frac{d(o,y)-a}{1+\alpha-\eps}. \]
If $d(o,y)\ge r$ then we obtain
\[ a \le \eps \cdot \frac{d(o,y)-a}{1+\alpha-\eps} \le \eps d(o,x) \]
and
\[ d(y,x) \le (\alpha-\eps) d(o,x) + a \le (\alpha-\eps) d(o,x) + \eps d(o,x) = \alpha d(o,y). \]
Therefore $S\setminus\OpenBall(o,r)\subseteq \alpha R$ and $\sigma^+(R,S)\le s^+(R,S)$.
\end{proof}

\begin{proposition}\label{proposition:premetric}
The function $s^+$ is a premetric on $\Unbdd$ satisfying a weak form of the triangle inequality, i.e.\ if $R,S,T$ are unbounded subsets of $X$, then
\begin{compactitem}
\item $s^+(R,S)\in[0,1]$ and $s^+(R,R)=0$,
\item $s^+(R,T)\le s^+(R,S)+s^+(R,S)s^+(S,T)+s^+(S,T)$.
\end{compactitem}
Similarly, $s$ is a symmetric premetric on $\Unbdd$ satisfying the same weak triangle inequality, i.e.
\begin{compactitem}
\item $s(R,S)\in[0,1]$ and $s(R,R)=0$,
\item $s(R,S) = s(S,R)$,
\item $s(R,T)\le s(R,S)+s(R,S)s(S,T)+s(S,T)$.
\end{compactitem}
Finally, $t$ is a pseudometric on $\Unbdd$, i.e.
\begin{compactitem}
\item $t(R,S)\in[0,1]$ and $t(R,R)=0$,
\item $t(R,S) = t(S,R)$,
\item $t(R,T)\le t(R,S)+t(S,T)$.
\end{compactitem}
\end{proposition}

\begin{proof}
The statements for $s^+$ and $s$ follow from the definition and from the Lemmas~\ref{lemma:greater-one} and \ref{lemma:transitive}.

It remains to show that $t$ satisfies the triangle inequality. Let $R,S$ be unbounded subsets of $X$. Then
\begin{align*}
s(R,T) &\le s(R,S)+s(R,S)s(S,T)+s(S,T) \\
	&\le s(R,S)+2\sqrt{s(R,S)s(S,T)}+s(S,T)
\end{align*}
which implies
\[ t(R,T)=\sqrt{s(R,T)}\le\sqrt{s(R,S)}+\sqrt{s(S,T)}=t(R,S)+t(S,T). \qedhere \]
\end{proof}

\begin{corollary}\label{corollary:well-def}
Assume that $R,S,T$ are unbounded subsets of $X$. If $s^+(S,T)=0$ then $s^+(R,T)\le s^+(R,S)$ and $s^+(S,R)\le s^+(T,R)$. Therefore, if $S\sim T$, then $s^+(R,S) = s^+(R,T)$, $s^+(S,R) = s^+(T,R)$, and $s(R,S) = s(R,T)$, $t(R,S)=t(R,T)$.
\end{corollary}

\begin{proof}
Using Proposition~\ref{proposition:premetric} and $s^+(S,T)=0$ we get
\[ s^+(R,T)\le s^+(R,S)+s^+(R,S)s^+(S,T)+s^+(S,T)=s^+(R,S) \]
and
\[ s^+(S,R)\le s^+(S,T)+s^+(S,T)s^+(T,R)+s^+(T,R)=s^+(T,R). \]
The remaining claims follow, since $S\sim T$ implies $s^+(S,T)=s^+(T,S)=0$.
\end{proof}

\begin{corollary}\label{corollary:equiv-rel}
Linear equivalence is an equivalence relation on unbounded subsets and the functions $s^+$, $s$, $t$ are well-defined on the quotient space $\Quotient\Unbdd$.
\end{corollary}

\begin{proof}
Reflexivity and symmetry follow immediately from the definition. Transitivity follows from Corollary~\ref{corollary:well-def}. Let $R_1,R_2,S_1,S_2$ be unbounded subsets and suppose $s(R_1,R_2)=s(S_1,S_2)=0$. Corollary~\ref{corollary:well-def} implies that $s^+(R_1,S_1)=s^+(R_2,S_1)=s^+(R_2,S_2)$, whence $s^+$ and therefore $s$, $t$ are well-defined on equivalence classes.
\end{proof}

\begin{theorem}\label{theorem:metric}
$(\Quotient\Unbdd,t)$ is a complete metric space.
\end{theorem}

\begin{proof}
By Proposition~\ref{proposition:premetric} and Corollary~\ref{corollary:equiv-rel} $(\Quotient\Unbdd,t)$ is a metric space. It remains to prove that it is also complete.

Let $(\xi_n)_{n\ge0}$ be a Cauchy sequence in $\Quotient\Unbdd$. Without loss of generality we may assume that $s(\xi_n,\xi_m)\le1/2$ for all $n,m$. Choose representatives $R_n\in\xi_n$. Then for any $\eps>0$ there is an index $N$ such that $s(R_n,R_m)<\eps$ for $n,m\ge N$. Therefore there exists a function $\eps^*\colon\N\to(0,1/2]$ such that $\eps^*$ is decreasing, $\eps^*(n)\to0$ as $n\to\infty$, and $s(R_m,R_n) < \eps^*(m)$ for $m\le n$.

According to Lemma~\ref{lemma:alternative} there are $r(m,n)\ge0$, for $m\le n$, such that
\[ R_n\setminus\OpenBall(o,r(m,n)) \subseteq \eps^*(m) R_m
	\qquad\text{and}\qquad
	R_m\setminus\OpenBall(o,r(m,n)) \subseteq \eps^*(m) R_n. \]
Hence there is an increasing function $r^*\colon\N\to[0,\infty)$ such that $r^*(n)\ge r(m,n)$ for $m\le n$. Applying Lemma~\ref{lemma:alternative} to $R_m\setminus\OpenBall(o,r^*(n))$ and $R_n\setminus\OpenBall(o,r^*(n))$ for $m\le n$ implies that there are $q(m,n)\ge0$ such that
\begin{align*}
R_n\setminus\OpenBall(o,q(m,n)) &\subseteq \eps^*(m) \bigl(R_m\setminus\OpenBall(o,r^*(n))\bigr), \\
R_m\setminus\OpenBall(o,q(m,n)) &\subseteq \eps^*(m) \bigl(R_n\setminus\OpenBall(o,r^*(n))\bigr).
\end{align*}
Thus there is an increasing function $q^*\colon\N\to[0,\infty)$ such that $q^*(n)\to\infty$ as $n\to\infty$
and $q^*(n)\ge q(m,n)$ for $m\le n$.

Let $x\in R_m$ with $q^*(n) \le d(o,x) < q^*(n+1)$ for some $n\ge m$. Then there is a $y\in R_n$ such that
\[ d(o,y)\ge r^*(n) \qquad\text{and}\qquad d(x,y) \le \eps^*(m) d(o,y). \]
Using the triangle inequality and $\eps^*(m) \le 1/2$ we get
\[ d(o,y)\le d(o,x) + d(x,y) \le d(o,x) + \eps^*(m) d(o,y) \le d(o,x) + d(o,y) / 2 \]
and
\begin{equation}\label{equation:estimate}
d(o,y)\le 2d(o,x) < 2q^*(n+1).
\end{equation}
We write $x^*$ to denote this element $y$ and define the set $S$ by
\[ S = \bigcup_{m\ge1} \{ x^* \setsep x\in R_m, \, d(o,x)\ge q^*(m) \}. \]
Then $S$ is an unbounded subset of $X$. Note that if $x\in S$ and $d(o,x)\ge 2q^*(m)$ for some $m$ then $x\in R_n$ for some $n\ge m$ due to the estimate in \eqref{equation:estimate}. We claim that $s(S,R_m)\le\eps^*(m)$ for $m\ge1$.
\begin{compactitem}
\item Let $x$ be an element of $R_m$ with $d(o,x)\ge q^*(m)$.
	Then, by construction of $S$, there is a $y\in S$ with $d(x,y)\le\eps^*(m) d(o,y)$.
	This implies
	\[ R_m\setminus\OpenBall(o,q^*(m)) \subseteq \eps^*(m) S. \]
\item Let $x$ be an element of $S$ with $d(o,x)\ge 2q^*(m)$.
	Then $x\in R_n$ for some $n\ge m$.
	This implies the lower bound $d(o,x)\ge r^*(n)$.
	By definition of $r^*$ there is a $y\in R_m$
	such that $d(x,y) \le \eps^*(m) d(o,y)$. Hence
	\[ S\setminus\OpenBall(o,2q^*(m)) \subseteq \eps^*(m) R_m. \]
\end{compactitem}
This implies the claim. Let $\zeta$ be the equivalence class of $S$. Then
\[ s(\xi_m,\zeta) = s(R_m,S) \le \eps^*(m) \]
for $m\ge1$. Therefore $\xi_m$ converges to $\zeta$ proving Cauchy completeness.
\end{proof}

\begin{definition}\label{definition:metric}
We call $t$ \emph{angle metric of unbounded sets} (see Example~\ref{example:euclid}).
\end{definition}

If $\Xi$ is a subset of $\Quotient\Unbdd$, we write $\closure(\Xi)$ to denote the closure of $\Xi$ in the metric space $(\Quotient\Unbdd,t)$. Let $\Elements\subseteq\Unbdd$ be a family of unbounded subsets of $(X,d)$. Define $\Quotient\Elements$ to be the set of equivalence classes in $\Quotient\Unbdd$ which contain at least one element from $\Elements$, this is
\[ \Quotient\Elements = \{ [R] \setsep R\in\Elements \} \subseteq \Quotient\Unbdd, \]
where $[R]$ is the equivalence class of $R$ with respect to linear equivalence $\linequiv$. Note that $\Bndry\Elements$ is a well-defined subset of $\Quotient\Unbdd$ which is closed and hence Cauchy complete. Thus up to isometry $(\Bndry\Elements,t)$ is the Cauchy completion of $(\Quotient\Elements,t)$.

\begin{remark}
The definition of the set $\Quotient\Elements$ depends on the underlying metric space $(X,d)$. However, no confusion should occur, since the underlying metric space will be clear from the context. Moreover, the above definition of $\Quotient\Elements$ is somewhat unusual, since $\Quotient\Elements\subseteq\Quotient\Unbdd$. The reason for this definition is that we will use topological notions of $(\Quotient\Unbdd,t)$ for the subset $\Quotient\Elements$. Furthermore, note that, if $\mathord\linequiv_\Elements$ denotes the restriction of $\linequiv$ to the set $\Elements$ then
\[ \Quotient\Elements \to \Elements/\mathord\linequiv_\Elements, \quad
	\zeta \mapsto \zeta\cap\Elements \]
is a canonical bijection.
\end{remark}

\begin{example}\label{example:euclid}
Consider $\R^n$ equipped with the usual $\ell^2$\ndash metric. For a nonzero vector $x\in\R^n$ let $L_x$ denote the line $\{\lambda x \setsep \lambda\in\R\}$ and $H_x$ the half-line $\{\lambda x \setsep \lambda\geq 0\}$. Set $\mathcal L=\{L_x \setsep x\in\R^n, x\ne0\}$ and $\mathcal H=\{H_x \setsep x\in\R^n, x\ne0\}$. Then $\Bndry{\mathcal L}$ is the projective space $\P^{n-1}$ and $\Bndry{\mathcal H}$ is the sphere $\S^{n-1}$. If $x,y\in\R^n\setminus\{0\}$ then
\[ s(L_x,L_y) = \sin(\angle(L_x,L_y)) \qquad\text{and}\qquad
	s(H_x,H_y) = \sin(\min\{\tfrac12\pi,\angle(H_x,H_y)\}) \]
where $\angle(L_x,L_y)$ is the smaller angle between the lines $L_x$ and $L_z$ and $\angle(H_x,H_y)$ is the angle between the half-lines $H_x,H_y$.
\end{example}

The following examples show that the function $s$ is not always a metric and that geodesics do not always yield a nice structure.

\begin{example}\label{example:nometric}
Consider the $2$\ndash dimensional space $\R^2$ with $\ell^1$\ndash metric $d_1$. Let $x_1=(1,0)$, $x_2=(2,1)$, and $x_3=(1,1)$. Set $L_i=\{ \lambda x_i \setsep \lambda\in\R \}$ for $i\in\{1,2,3\}$. Then
\[ s(L_1,L_2)=\tfrac12, \qquad s(L_2,L_3)=\tfrac13, \qquad s(L_1,L_3)=1, \]
and the triangle inequality is not satisfied.
\end{example}

\begin{example}\label{example:z2boundary}
Consider the metric space $(\Z^2,d_1)$, where $d_1$ is the $\ell^1$\ndash metric. In this discrete setting a geodesic ray is an infinite sequence $(x_0,x_1,\dotsc)$ in $\Z^2$ such that $d(x_i,x_j) = \abs{i-j}$. Let $\mathcal G$ be the family of all geodesic rays emanating from the origin. Furthermore, let $\Elements$ be the family of all sets $\{n x \setsep n\in\N_0\}$ for $x\in\Z^2$, $x\ne0$. Then the space $\Bndry{\mathcal G}$ contains much more elements than $\Bndry\Elements$. To see this set $x_{2n}=(2^n-1,2^n-1)$ and $x_{2n+1}=(2^{n+1}-1,2^n-1)$ for $n\in\N_0$. Join $x_m$ and $x_{m+1}$, $m\in\N_0$,  by a geodesic path and let $R$ denote the ray consisting of the union of these finite geodesic paths. Obviously $R$ is a geodesic ray and there is some $\eps>0$ such that $s(R,S)\ge\eps$ for all $S\in\Elements$.
\end{example}

\section{Quasi-isometries}
\label{section:quasi-iso}

\begin{definition}\label{definition:quasi-iso}
Let $(X,d_X)$ and $(Y,d_Y)$ be metric spaces and let $q>0$. A function $f\colon X \to Y$ is called a \emph{$q$\ndash quasi-isometry} if
\[ q^{-1} d_X(x,x')-q \le d_Y(f(x),f(x')) \le q d_X(x,x)+q \]
for all $x,x'\in X$ and such that every closed ball in $Y$ with radius $q$ contains an element of $f(X)$. We say that two metrics $d_1$ and $d_2$ on $X$ are \emph{quasi-isometrically equivalent}, if the identity is a quasi-isometry from $(X,d_1)$ to $(X,d_2)$.
\end{definition}

\begin{lemma}\label{lemma:qi-cones}
Let $f\colon X \to Y$ be a $q$\ndash quasi-isometry of the metric spaces $(X,d_X)$ and $(Y,d_Y)$. Let $R,S$ be unbounded subsets of $X$. Then $f(R),f(S)$ are unbounded subsets of $Y$ and
\[ q^{-2} s_X^+(R,S) \le s_Y^+(f(R),f(S)) \le q^2 s_X^+(R,S). \]
\end{lemma}

\begin{proof}
We fix reference points $o$ and $f(o)$ in $X$ and $Y$, respectively. First of all note that
\[ q^{-1} d_X(x,x')-q \le d_Y(f(x),f(x')) \le q d_X(x,x')+q \]
implies
\[ q^{-1} d_Y(f(x),f(x')) - 1 \le d_X(x,x') \leq q d_Y(f(x),f(x')) + q^2 \]
for all $x,x'\in X$.

Let $\alpha>s_X^+(R,S)$. Then there is a number $a$ such that $S\subseteq \alpha R+a$. Hence for $x'\in S$ there is a point $x\in R$ with $d_X(x',x) \le \alpha d_X(o,x)+a$. Since $f$ is a $q$\ndash quasi-isometry, we get $d_X(o,x) \le q d_Y(f(o),f(x)) + q^2$. This implies
\begin{align*}
d_Y(f(x'),f(x)) &\le q d_X(x',x) + q \le q\alpha d_X(o,x) + qa+q \\
	&\le q^2\alpha d_Y(f(o),f(x)) + q^3\alpha+qa+q,
\end{align*}
proving that
\[ f(S) \subseteq q^2\alpha f(R) + q^3\alpha+qa+q \]
holds. Thus $s_Y^+(f(R),f(S)) \le q^2 s_X^+(R,S)$.

If $s_X^+(R,S)=0$ then $s_Y^+(f(R),f(S))\ge q^{-2} s_X^+(R,S)$ trivially holds. Hence we assume that $s_X^+(R,S)>0$. Then, for $\alpha<s_X(R,S)$, $S\subseteq \alpha R+a$ fails to be true for all $a\ge0$. Hence for every $a\ge0$ there exists a point $x'\in S$ which is not contained in $\alpha R+a$. Thus $d_X(x',x)>\alpha d(o,x)+a$ for all $x\in R$. This implies
\begin{align*}
d_Y(f(x'),f(x)) &\ge q^{-1} d_X(x',x) - q > q^{-1}\alpha d_X(o,x) + q^{-1}a-q \\
	&\ge q^{-2}\alpha d_Y(f(o),f(x)) + q^{-1}(a-\alpha)-q.
\end{align*}
Thus $f(x')$ is not contained in $q^{-2}\alpha f(R) + q^{-1}(a-\alpha)-q$. Since $a\ge0$ was arbitrary, this means that $s_Y^+(f(R),f(S))\ge q^{-2} s_X^+(R,S)$.
\end{proof}

\begin{theorem}\label{theorem:qi-bilip}
Let $f\colon X \to Y$ be a $q$\ndash quasi-isometry of the metric spaces $(X,d_X)$ and $(Y,d_Y)$. Then $f$ induces a bijection $f\colon\Quotient{\Unbdd_X}\to\Quotient{\Unbdd_Y}$ which is bi-Lipschitz continuous:
\[ q^{-1} t_X(\zeta,\xi) \le t_Y(f(\zeta),f(\xi)) \le q t_X(\zeta,\xi) \]
for all $\zeta,\xi\in \Quotient{\Unbdd_X}$. In particular, if $\Elements$ is a family of unbounded subsets in $X$, then $f(\Quotient\Elements) = \Quotient{f(\Elements)}$ and $f(\Bndry\Elements) = \Bndry{f(\Elements)}$.
\end{theorem}

\begin{proof}
Of course $f(\Unbdd_X)$ is a subset of $\Unbdd_Y$. By Lemma~\ref{lemma:qi-cones} the function $f\colon\Quotient{\Unbdd_X}\to\Quotient{\Unbdd_Y}$ which maps the equivalence class of an unbounded $R\subseteq X$ to the equivalence class of $f(R)$ is well-defined, one-to-one, and satisfies
\[ q^{-1} t_X(\zeta,\xi) \le t_Y(f(\zeta),f(\xi)) \le q t_X(\zeta,\xi) \]
for all $\zeta,\xi\in \Quotient{\Unbdd_X}$. Thus it remains to show that $f$ is also onto. Let $S$ be an unbounded subset of $Y$. Since $f$ is a $q$\ndash quasi-isometry, the set $R=f^{-1}((0S+q)\cap f(X))$ is an unbounded subset of $X$ and $f(R) = (0S+q)\cap f(X) \linequiv S$.
\end{proof}

\section{Boundary at infinity and angular metric in a \texorpdfstring{$\CAT(0)$}{CAT(0)}~space}
\label{section:cat}

Let us recall the definitions of the boundary at infinity and the angular metric in a $\CAT(0)$~space. For more details we refer to the book of Bridson and Haefliger~\cite{bridson1999metric}, see especially Chapter~II.8 and Chapter~II.9 therein. A \emph{geodesic ray} in a metric space $(X,d)$ is a curve $c\colon[0,\infty)\to X$ such that $d(c(x),c(y))=\abs{x-y}$ for all $x,y\ge 0$. The \emph{boundary at infinity} $\partial X$ of $X$ is defined to be the set of equivalence classes of geodesic rays, where geodesic rays $c,c'$ are equivalent whenever they stay at bounded distance, that is, if there is a constant $K$, such that $d(c(x),c'(x))\le K$ for all $x\in[0,\infty)$. In the sequel we assume that $X$ is a complete $\CAT(0)$~space.

For each point $p$ in $X$ and $\xi$ in $\partial X$ there is precisely one geodesic ray belonging to $\xi$ which emanates from $p$. Then $\angle_p(\xi,\zeta)$ for $\xi,\zeta\in\partial X$ is defined to be the angle at $p$ between the uniquely determined rays in $\xi$ and $\zeta$ which emanate from $p$. The angle between $\xi$ and $\zeta$ is defined by
\[ \angle(\xi,\zeta) = \sup\{\angle_p(\xi,\zeta) \setsep p\in X\}. \]
This yields a metric on $\partial X$ called \emph{angular metric} and $(\partial X, \angle)$ is a complete metric space. For our purposes the following description of the angular metric is useful. Fix a reference point $o$ in $X$. If $\xi\in\partial X$, we write $c_\xi$ for the uniquely determined geodesic ray in $\xi$ which emanates from $o$ and $R_\xi$ for the image of $c_\xi$ in $X$, i.e.\ $R_\xi = c_\xi([0,\infty))$. Then, see \cite[Proposition~9.8~(4)]{bridson1999metric},
\[ 2\sin\bigl(\tfrac12\angle(\xi,\zeta)\bigr) = \lim_{x\to\infty} \tfrac1x d(c_\xi(x),c_\zeta(x)). \]

\begin{lemma}\label{lemma:comp}
Let $\xi,\zeta$ be elements in $\partial X$. Then
\[ s(R_\xi,R_\zeta) \le 2\sin\bigl(\tfrac12\angle(\xi,\zeta)\bigr) \le 4s(R_\xi,R_\zeta). \]
\end{lemma}

\begin{proof}
Note that $d(o,c_\xi(x)) = d(o,c_\zeta(x)) = x$ for all $x\in[0,\infty)$, since $c_\xi(0)=c_\zeta(0)=o$. Suppose that $\alpha>2\sin\bigl(\tfrac12\angle(\xi,\zeta)\bigr)$. Then there exists a constant $a\ge0$, such that $d(c_\xi(x),c_\zeta(x)) \le \alpha x$ for all $x\ge a$. This implies that
\[ R_\xi \subseteq \alpha R_\zeta + a \qquad\text{and}\qquad R_\zeta \subseteq \alpha R_\xi + a. \]
Therefore $s(R_\xi,R_\zeta)\le\alpha$ which yields the lower bound.

If $s(R_\xi,R_\zeta)\ge\tfrac12$ then the upper bound is trivially true. Hence assume that $s(R_\xi,R_\zeta)<\tfrac12$ and fix some $\alpha$, such that $s(R_\xi,R_\zeta)<\alpha\le\tfrac12$. By Lemma~\ref{lemma:alternative} there is a constant $r\ge0$, such that $R_\zeta\setminus U(o,r)\subseteq \alpha R_\xi$. Hence, for any $x\ge r$, there is a $y=y(x)\ge0$, such that
\[ d(c_\zeta(x),c_\xi(y)) \le \alpha d(o,c_\xi(y)) = \alpha y. \]
Using the triangle inequality the estimate above yields $y\le x+\alpha y$ and $x\le y+\alpha y$. It follows that $\abs{y-x}\le\alpha y$ and $y\le2x$, since $\alpha\le\tfrac12$. Collecting the pieces we get
\begin{align*}
d(c_\zeta(x), c_\xi(x))
&\le d(c_\zeta(x), c_\xi(y)) + d(c_\xi(y), c_\xi(x)) \\
&\le \alpha d(o, c_\xi(y)) + \abs{y-x} \\
&\le 2\alpha y \le 4\alpha x.
\end{align*}
and thus
\[ 2\sin\bigl(\tfrac12\angle(\xi,\zeta)\bigr)
	= \lim_{x\to\infty} \tfrac1x d(c_\zeta(x), c_\xi(x))
	\le 4\alpha. \qedhere \]
\end{proof}

As a consequence of the previous lemma we get that two geodesic rays $c$ and $c'$ stay at bounded distance if and only if the subsets $c([0,\infty))$ and $c'([0,\infty))$ are linearly equivalent. Write $\mathcal G$ to denote the family of all subsets of the form $c([0,\infty))$, where $c$ is some geodesic ray in $X$.

\begin{proposition}\label{proposition:realize}
Let $X$ be a complete $\CAT(0)$~space and equip $\partial X$ with the angular metric $\angle$. Then
\[ \partial X \to \Quotient{\mathcal G}, \quad \xi \mapsto [R_\xi], \]
where $[R_\xi]$ is the equivalence class of $R_\xi$ with respect to linear equivalence, is one-to-one, onto, and bi-H\"older continuous:
\[ \tfrac1\pi\angle(\xi,\zeta) \le \bigl( t([R_\xi],[R_\zeta]) \bigr)^2 \le \angle(\xi,\zeta) \]
for all $\xi,\zeta\in\partial X$. Furthermore, $\Quotient{\mathcal G}$ is a closed subset of $(\Quotient\Unbdd,t)$, since $(\partial X,\angle)$ is a complete metric space.
\end{proposition}

\begin{proof}
Since $\angle(\xi,\zeta)\in[0,\pi]$ and $\tfrac2\pi x \le \sin(x) \le x$ for all $x\in[0,\tfrac\pi2]$, Lemma~\ref{lemma:comp} yields $\tfrac1\pi \angle(\xi,\zeta) \le s(R_\xi,R_\zeta) \le \angle(\xi,\zeta)$
\end{proof}

\section{Boundaries of groups}
\label{section:groups}

Let $G$ be a group and $d$ be a metric on $G$. Fix the identity element $1\in G$ as reference point. From an algebraic point of view it is natural to consider the families of unbounded cyclic subgroups and unbounded cyclic subsemigroups of the group $G$. Hence define
\[ \Orbit G = \{ \gen g \setsep g\in G,\, \gen g\in\Unbdd \} \]
and
\[ \Forward G = \{ \genp g \setsep g\in G,\, \genp g\in\Unbdd \}, \]
where $\genp g = \{ g^n \setsep n\in\N_0 \}$ is the semigroup generated by $g\in G$. Note that, if $\genp g \in \Forward G$, then $\gen g \in \Orbit G$.

\begin{definition}\label{definition:groupbndry}
We define the \emph{projective boundary} of $G$ by $\Project G=\Bndry{\Orbit G}$ and the \emph{linear boundary} by $\Linear G=\Bndry{\Forward G}$
\end{definition}

\begin{remark}
Both, $\Project G$ and $\Linear G$, depend on the metric $d$. If it is necessary to emphasize this dependence, we write $\Project(G,d)$ and $\Linear(G,d)$, respectively.
\end{remark}

\begin{lemma}\label{lemma:orbits}
If $g\in\Orbit G$ and $h\in\Forward G$ then $\gen{g^n}\linequiv\gen g$ and $\genp{h^n}\linequiv\genp h$ for all $n\in\N$. Furthermore, if $d$ is left-invariant or right-invariant, then $\genp g \in \Forward G$ if and only if $\gen g \in \Orbit G$.
\end{lemma}

There are two interesting sources for metrics on a group $G$. If $G$ is finitely generated (or more generally compactly generated), it is natural to consider word metrics on $G$. If $G$ is a connected Lie group, it is natural to consider left-invariant Riemannian metrics on $G$. In this case $G$ is also compactly generated and Corollary~\ref{corollary:milnor2} implies that any left-invariant Riemannian metric is quasi-isometrically equivalent to any word metric on $G$. Hence for our purposes it is sufficient to study the setting of compactly generated groups in more detail.

A topological group is called \emph{compactly generated}, if there is a compact generating set $K\subseteq G$. In this case $S=K\cup K^{-1}$ is a compact, symmetric (i.e.\ $S=S^{-1}$), generating set. Set $S^0=\{1\}$ and $S^n = \{s_1\dotsm s_n \setsep s_1,\dotsc, s_n\in S\}$ for $n\ge1$. Note that $S^n$ is compact and symmetric for all $n\ge0$ and
\[ G = \bigcup_{n\ge0} S^n. \]
The \emph{word metric} $d$ of $G$ with respect to $S$ is defined by $d(g,h) = \inf\{n \setsep g^{-1}h \in S^n\}$. The metric $d$ is left-invariant and induces the discrete topology on $G$ which is in general different from the group topology. In the sequel we consider the class of compactly generated, locally compact Hausdorff groups. Some facts about such groups and their word metrics are provided by Appendix~\ref{appendix:compact}. Finitely generated groups fit in this setting (in this case a finitely generated group is equipped with the discrete topology). If not stated otherwise, all topological notions refer to the group topology (except for boundedness which refers to the word metric $d$).

We fix some compactly generated, locally compact Hausdorff group $G$ and a word metric $d$ on $G$. Notice that a subset of $G$ is bounded with respect to $d$ if and only if it is relatively compact (see Lemma~\ref{lemma:basics}). Suppose that $d'$ is another word metric on $G$ or (more general) a metric which is quasi-isometrically equivalent to $d$. Then, by Theorem~\ref{theorem:qi-bilip}, $t_{(G,d)}$ and $t_{(G,d')}$ are bi-Lipschitz-equivalent. Hence linear equivalence and all notions which only depend on the topological or uniform structure of $\Quotient\Unbdd$ (like closure or Cauchy completeness for instance), do not depend on the generating set. In particular, we obtain the following statement.

\begin{lemma}\label{lemma:boundary}
If a compactly generated, locally compact Hausdorff group $G$ is equipped with a word metric $d$ then the (topological) spaces $\Linear G$ and $\Project G$ do not depend on the choice of the word metric (or of the generating set).
\end{lemma}

A group element $g$ is called \emph{compact}, if $\gen g$ is relatively compact, and \emph{non-compact} otherwise. Thus $g$ is non-compact if and only if $\gen g\in\Orbit G$. Notice that, if $G$ is finitely generated, a group element $g$ is non-compact if and only if $g$ is non-torsion. Furthermore, by Weil's lemma (see \cite[Theorem~9.1]{hewitt1979abstract}) $g$ is non-compact, if and only if $\gen g$ is the image of a monomorphism $\Z\to G$ which is a topological isomorphism onto $\gen g$ (a topological isomorphism is a group isomorphism which is also a homeomorphism). Hence, Weil's lemma implies the following.

\begin{lemma}\label{lemma:unbdd}
If $g\in G$ is non-compact then $d(1,g^n)\to\infty$ for $n\to\infty$.
\end{lemma}

\begin{remark}
Notice, that compact group elements of a group $G$ do not contribute to the boundaries $\Project G$ and $\Linear G$. Especially, if $G$ only contains compact group elements, then these boundaries are empty. In the discrete case this means that torsion groups have empty boundaries.
\end{remark}

\begin{remark}
Let $g$ and $h$ be non-compact group elements. We have seen that $s(\genp g, \genp h)\le 1$ and $s(\genp g, \genp{g^n})=0$ for all $n\in\N$. Now it is natural to ask, what can be said about $s(\genp g, \genp{g^{-1}})$. Often $s(\genp g, \genp{g^{-1}})=1$, but in \cite{kroen2012linear} Kr\"on, Lehnert and Stein give an example of a finitely generated group constructed by iterated HNN-extensions with a non-torsion element $g$ such that $s(\genp g, \genp{g^{-1}})\le\frac{12}{17}$. They also show that in general this value cannot be arbitrarily close to zero. Indeed, $s(\genp g, \genp{g^{-1}})$ is always greater or equal $\frac12$. The infimum of these values (for all groups) is unknown. In \cite{kroen2012linear} there is also an example of a finitely generated group with non-torsion elements $g,h$ for which $\genp g \linequiv \genp h$ but $\genp{g^{-1}} \not\linequiv \genp{h^{-1}}$.
\end{remark}

The following lemma yields a useful alternative to compute $s^+(\genp g, \genp h)$ and $s^+(\gen g, \gen h)$.

\begin{lemma}\label{lemma:limsup}
Let $g$ and $h$ be non-compact group elements. Then
\[ s^+(\genp g, \genp h)
	= \limsup_{n\to\infty} \, \inf\biggl\{ \frac{d(h^n,g^m)}{d(1,g^m)} \setsep m\in\N_0 \biggr\} \]
and
\[ s^+(\gen g, \gen h)
	= \limsup_{\abs{n}\to\infty} \, \inf\biggl\{ \frac{d(h^n,g^m)}{d(1,g^m)} \setsep m\in\Z \biggr\}. \]
\end{lemma}

\begin{proof}
We only prove the first claim, since the proof of the second is analogous. Suppose that $\alpha>s^+(\genp g, \genp h)$. Hence $\genp h \subseteq \alpha \genp h + a$ for some $a\ge0$. Thus, for each $n\in\N_0$, there is an integer $k=k(n)\ge0$, such that $d(h^n,g^k) \le \alpha d(1,g^k) + a$. Then
\[ \inf\biggl\{ \frac{d(h^n,g^m)}{d(1,g^m)} \setsep m\in\N_0 \biggr\}
	\le \frac{d(h^n,g^k)}{d(1,g^k)}
	\le \alpha + \frac{a}{d(1,g^k)}. \]
Using the triangle inequality we get
\[ (1-\alpha) d(1,g^k) - a \le d(1,h^n) \le (1+\alpha) d(1,g^k) + a. \]
If $n\to\infty$ then $d(1,h^n)\to\infty$ by Lemma~\ref{lemma:unbdd} and therefore $d(1,g^k)\to\infty$. This implies
\[ \limsup_{n\to\infty} \inf\biggl\{ \frac{d(h^n,g^m)}{d(1,g^m)} \setsep m\in\N_0 \biggr\}
	\le \limsup_{n\to\infty} \alpha + \frac{a}{d(1,g^k)} = \alpha. \]
In order to prove the reversed inequality assume that
\[ \alpha > \limsup_{n\to\infty} \inf\biggl\{ \frac{d(h^n,g^m)}{d(1,g^m)} \setsep m\in\N_0 \biggr\}. \]
Then there is an integer $N\ge0$, such that
\[ \inf\biggl\{ \frac{d(h^n,g^m)}{d(1,g^m)} \setsep m\in\N_0 \biggr\} \le \alpha \]
for all $n\ge N$. Let $\eps>0$. Then, for each $n\ge N$, we can find an integer $k=k(n)\ge0$, such that
\[ \frac{d(h^n,g^k)}{d(1,g^k)} \le \alpha+\eps. \]
Set $a=\max\{d(1,h^n) \setsep 0\le n<N\}$. Then we obtain $\genp h \subseteq (\alpha+\eps) \genp g + a$.
\end{proof}

Using Corollary~\ref{corollary:milnor1} and its notation, we obtain the following:

\begin{lemma}\label{lemma:finite}
Let $G$ be a compactly generated, locally compact group. The following statements are true up to bi-Lipschitz-equivalence of the metric $t$:
\begin{itemize}
\item Suppose that $N$ is a compact group and $H$ is a topological Hausdorff group.
	If $\{1\} \longrightarrow N \longrightarrow H \stackrel{\pi}{\longrightarrow} G \longrightarrow \{1\}$
	is a topological short exact sequence, such that $\pi\colon H\to G$ is also open,
	then $H$ and $G$ have the same linear and projective boundaries, respectively.
\item If $H$ is a closed subgroup of $G$ and $(H\rightfactor G, d_{H\rightfactor G})$ is bounded then
	\[ \Linear H \subseteq \Linear G \qquad\text{and}\qquad
		\Project H \subseteq \Linear G. \]
	If $H$ is of finite index in $G$ then equality holds.
\end{itemize}
\end{lemma}

\begin{proof}
In order to prove the first statement note that, by Corollary~\ref{corollary:milnor1} the homomorphism $\pi\colon H\to G$ is a quasi-isometry. Assume that $h\in H$ and $\gen{\pi(h)}$ is bounded in $G$ then $\pi^{-1}(\gen{\pi(h)})$ is bounded by Lemma~\ref{lemma:maps}. Hence $\gen h\subseteq \pi^{-1}(\gen{\pi(h)})$ is bounded. Thus unbounded cyclic sub\-\hbox{(semi)}\-groups of $H$ are mapped onto unbounded cyclic sub\-\hbox{(semi)}\-groups of $G$. This implies the first statement using Theorem~\ref{theorem:qi-bilip}.

Now suppose that $H$ is a closed subgroup of $G$ and $H\rightfactor G$ is bounded. By Corollary~\ref{corollary:milnor1} the inclusion is a quasi-isometry. In order to emphasize the dependence on $H$ and $G$, we use subscripts $H$ and $G$. By Theorem~\ref{theorem:qi-bilip} we have
\[ \Linear H = \closure_H(\Forward H/\mathord{\linequiv_H})
	= \closure_G(\Forward H/\mathord{\linequiv_G})
	\subseteq \closure_G(\Forward G/\mathord{\linequiv_G}) = \Linear G \]
and analogously for $\Project H\subseteq \Project G$. Assume that $H$ has finite index in $G$. If $g\in G$ then $H\rightfactor H\gen g$ is finite. Thus there are $k\in \Z$ and $n>0$, such that $Hg^{k+n} = Hg^k$. This implies $g^n\in H$. Hence in this case Lemma~\ref{lemma:orbits} implies
\[ \Forward H/\mathord{\linequiv_H} = \Forward G/\mathord{\linequiv_G}
	\qquad\text{and}\qquad \Orbit H/\mathord{\linequiv_H} = \Orbit G/\mathord{\linequiv_G} \]
which yields the assertion.
\end{proof}

In the setting of finitely generated groups the previous lemma implies that two weakly commensurable finitely generated groups $G$ and $H$ (i.e.\ there is a group $Q$ and homomorphisms $Q\to G$ and $Q\to H$ both having finite kernels and images of finite index) have the same linear and projective boundaries. In the continuous setting the situation is more complicated: In general it is possible that
\[ \Forward H/\mathord{\linequiv_H} \subsetneq \Forward G/\mathord{\linequiv_G} \]
(consider for instance $\Z^2\le\R^2$). However, equality may hold after taking closures on both sides, i.e.\ $\Linear H = \Linear G$. The problem here is to find for each non-compact $g\in G$ a sequence $(h_n)_{n\ge0}$ in $H$, such that $t(\genp g, \genp{h_n}) \to 0$ for $n\to\infty$. Notice that there is always an unbounded subset $R\subseteq H$ with $\genp g \linequiv R$, if $H\rightfactor G$ is bounded.

With this preparations we can settle the commutative case completely. Recall that, if $G$ is a commutative, compactly generated, locally compact Hausdorff group, then by \cite[Theorem~9.8]{hewitt1979abstract} there are integers $a,b\ge0$ and a commutative, compact Hausdorff group $C$, such that $G$ is topologically isomorphic to $\R^a\times\Z^b\times C$.

\begin{corollary}\label{corollary:commutative}
Assume that $G$ is a commutative, compactly generated, locally compact Hausdorff group.
\begin{compactitem}
\item If $G$ is topologically isomorphic to $\R^a\times\Z^b\times C$
	for some integers $a,b\ge0$ and some compact, commutative group $C$
	then $\Linear G = \S^{a+b-1}$ and $\Project G = \P^{a+b-1}$.
\item If $H$ is a closed subgroup, such that $G/H$ is compact
	 then $\Linear H = \Linear G$ and $\Project H = \Project G$.
\end{compactitem}
\end{corollary}

\begin{proof}
As $\R^a\times\Z^b$ is a quotient of $G$ with compact kernel, the linear and projective boundaries of $G$ and $\R^a\times\Z^b$ are the same, respectively. Since any word metric on $\R^a\times\Z^b$ is quasi-isometrically equivalent to the $\ell^2$\ndash metric on $\R^a\times\Z^b$, we may use the $\ell^2$\ndash metric. It is then easy to see that $\R^a\times\Z^b$ and $\R^{a+b}$ have the same boundaries. Hence the assertion follows from Example~\ref{example:euclid}.

Suppose that $H$ is a closed subgroup, such that $G/H$ is compact. As before, let $G$ be topologically isomorphic to $\R^a\times\Z^b\times C$. It follows that $H$ is topologically isomorphic to $\R^{a-c}\times\Z^{b+c}\times D$ for some integer $c$ and some commutative, compact Hausdorff group $D$. Thus the first assertion implies the second.
\end{proof}

In the setting of topological groups it is natural to consider also unbounded one-parameter subgroups and unbounded one-parameter subsemigroups, as well. A \emph{one-parameter subgroup} in $G$ is the image of a continuous homomorphism $\R\to G$ and a \emph{one-parameter subsemigroup} is the image of a continuous semigroup homomorphism $[0,\infty)\to G$. Obviously, if $\phi\colon[0,\infty)\to G$ is a continuous semigroup homomorphism, then there is a canonical extension to a continuous group homomorphism $\bar\phi\colon\R\to G$ and $\phi$ has unbounded image, if and only if $\bar\phi$ has. Define $\Orbit_\R G$ and $\Forward_\R G$ to be the family of unbounded one-parameter subgroups and unbounded one-parameter subsemigroups, respectively. Again, by Weil's lemma a continuous homomorphism $\phi\colon\R\to G$ has unbounded image if and only if $\phi$ is a topological isomorphism onto its image.

\begin{lemma}\label{lemma:orbits2}
Suppose that $\phi\colon\R\to G$ is a continuous homomorphism with unbounded image. Then
\[ \phi([0,\infty)) \linequiv \genp{\phi(t)} \qquad\text{and}\qquad \phi(\R) \linequiv \gen{\phi(t)} \]
for all $t>0$. Hence
\[ \Quotient{\Orbit_\R G} \subseteq \Quotient{\Orbit G}
	\qquad\text{and}\qquad
	\Quotient{\Forward_\R G} \subseteq \Quotient{\Forward G}. \]
\end{lemma}

\begin{proposition}\label{proposition:oneparam}
Let $G$ be a connected, nilpotent Lie group. Then
\[ \Quotient{\Orbit_\R G} = \Quotient{\Orbit G}
	\qquad\text{and}\qquad
	\Quotient{\Forward_\R G} = \Quotient{\Forward G}. \]
\end{proposition}

\begin{proof}
Let $\LieG$ be the Lie algebra of $G$ and $\exp\colon\LieG\to G$ be the exponential map. Then $\exp$ is surjective. Thus, if $g$ is a non-compact group element, then there is an element $x\in\LieG$ with $\exp(x)=g$. Then $\R\to G$, $t\mapsto\exp(tx)$ is a continuous homomorphism with unbounded image which proves the statement.
\end{proof}

\section{Boundaries of nilpotent groups}
\label{section:nilpotent}

In the following we determine the linear and projective boundary of connected, nilpotent Lie groups and their discrete counterparts, finitely generated nilpotent groups. A commutative, connected Lie group $G$ is isomorphic to $\R^a\times (\R/\Z)^b$ for some integers $a$ and $b$. In analogy to the discrete case we call the integer $a$ the \emph{compact-free dimension} of $G$. For convenience we define $\S^{-1}$ and $\P^{-1}$ to be the empty set.

\begin{theorem}\label{theorem:nilpotent}
Let $G$ be a nilpotent group which is either a connected Lie group or a finitely generated group. Suppose that $G$ has descending central series
\[ G = G_1 \supseteq G_2 \supseteq \dotsb \supseteq G_c \supsetneq G_{c+1} = \{1\}, \]
where $c\ge1$ is the nilpotency class of $G$. Let $\nu(i)$ denote the compact-free dimension or torsion-free rank of $G_i/G_{i+1}$, respectively. Then the linear boundary $\Linear G$ is homeomorphic to the disjoint union of $c$ spheres:
\[ \Linear G = \S^{\nu(1)-1} \uplus \S^{\nu(2)-1}\uplus \dotsb \uplus \S^{\nu(c)-1}. \]
Analogously, the projective boundary $\Project G$ is homeomorphic to the disjoint union of projective spaces:
\[ \Project G = \P^{\nu(1)-1} \uplus \P^{\nu(2)-1}\uplus \dotsb \uplus \P^{\nu(c)-1}. \]
\end{theorem}

If two finitely generated, nilpotent groups $G$ and $H$ are weakly commensurable then previous result yields a new proof of the fact that the multisets
\[ \{ \nu_1(G),\nu_2(G),\dotsc \} \qquad\text{and}\qquad \{ \nu_1(H),\nu_2(H),\dotsc \} \]
of torsion-free ranks are equal, since the boundaries of $G$ and $H$ are bi-Lipschitz-equivalent. Notice that there is no information on the ordering and it is unclear, whether it is possible to deduce the ordering from the angle metrics of $G$ and $H$, respectively . It is a corollary of Pansu's theorem (see \cite[Th\'eor\`eme~3]{pansu1989metriques}), that even the tuples
\[ (\nu_1(G),\nu_2(G),\dotsc) \qquad\text{and}\qquad (\nu_1(H),\nu_2(H),\dotsc) \]
are equal.

First we prove the theorem for connected Lie groups and then use the Mal'tsev completion to deduce the statement for finitely generated groups. For both cases we use the notation and results of Appendix~\ref{appendix:groups}.

\begin{proof}[Proof of Theorem~\ref{theorem:nilpotent} in the Lie case]
We prove that statement for $\Linear G$, as the other case is completely analogous. Let $G$ be a connected, nilpotent Lie group with word metric $d_G$. Set $t_G=t_{(G,d_G)}$ and write $\linequiv_G$ to denote linear equivalence in $(G,d_G)$. By Lemma~\ref{lemma:ranks} and Lemma~\ref{lemma:finite} we may assume that $G$ is also simply connected. Set $t_a=t_{(\LieG,d_a)}$ and write $\linequiv_a$ to denote linear equivalence in $(\LieG,d_a)$. By the Lemmas \ref{lemma:compare}, \ref{lemma:additive}, \ref{lemma:three} the map
\[ \phi \colon \Forward(\LieG,+)/\mathord{\linequiv_a} \to
	\Forward(G,\cdot)/\mathord{\linequiv_G} \]
which maps the equivalence class of $\genp x\in\Forward(\LieG,+)$ to the equivalence class of $\genp{\exp(x)}\in\Forward(G,\cdot)$, is well-defined and bi-H\"older continuous with respect to the metrics $t_a$ and $t_G$, respectively. Hence $\phi$ extends to a bi-H\"older continuous map from $\Linear(\LieG,d_a)$ to $\Linear(G,d_G)$. Then the assertion follows from the first part of Lemma~\ref{lemma:additive}.
\end{proof}

\begin{proof}[Proof of Theorem~\ref{theorem:nilpotent} in the discrete case]
Let $\Gamma$ be a finitely generated, nilpotent group. We only show the assertion for $\Linear\Gamma$ for the same reason as above. By Lemma~\ref{lemma:ranks} and Lemma~\ref{lemma:finite} we may assume that $\Gamma$ is also torsion-free. Then the (real) Mal'tsev completion of $\Gamma$ yields a connected, simply connected, nilpotent Lie group $G$, such that $\Gamma$ is a uniform subgroup of $G$, see \cite{maltsev1951class}. Using Lemma~\ref{lemma:finite} it follows that $\Linear\Gamma\subseteq\Linear G$. Let $d_G$ be a word metric on $G$ and set $t_G=t_{(G,d_G)}$. In order to prove equality, it is sufficient to construct for each $g\in G$ a sequence $h_1,h_2,\dotsc\in\Gamma$, such that $t_G(\genp g, \genp{h_m})\to0$ if $m\to\infty$. Suppose that $g$ is an element of $G_n$. Set $\Lambda=\log(\Gamma)$ and $x=\log(g)$. Then $\Lambda\cap\LieG_k$ is a uniform subgroup in $(\LieG_k,\cdot)$ for all $k$. Hence $\pi_n(\Lambda\cap\LieG_n)$ is a uniform subgroup of $(V_n,+)$, since $\pi_n$ is a continuous epimorphism from $(\LieG_n,\cdot)$ to $(V_n,+)$. As $V_n$ is isomorphic to $\R^{\nu(n)}$, $\pi_n(\Lambda\cap\LieG_n)$ is isomorphic to $\Z^{\nu(n)}$. Thus there is a sequence $y_1,y_2,\dotsc\in\Lambda\cap\LieG_n$, such that $t_a(\genp{\pi(x)}, \genp{\pi(y_m)})\to0$ if $m\to\infty$. Since $\genp x\linequiv_a \genp{\pi(x)}$ and $\genp{y_m}\linequiv_a\genp{\pi(y_m)}$, we infer that $t_a(\genp x, \genp{y_m})\to0$ if $m\to\infty$. Set $h_m=\exp(z_m)\in\Gamma$. Then $t_G(\genp g, \genp{h_m})\to0$ for $m\to\infty$ using Lemma~\ref{lemma:three} as required.
\end{proof}

\begin{remark}
We have carried out an alternative proof for the discrete case which avoids the use of Mal'tsev completion and tools from Lie theory and employs techniques from combinatorial group theory---mainly commutator calculus and careful analysis of word lengths'. This proof follows similar lines compared to the proof for the Lie case given here.
\end{remark}

\section{Boundaries of vertex-transitive graphs with polynomial growth}
\label{section:graph-boundaries}

Let $G$ be a group and let $S$ be a finite generating set of $G$. Then the Cayley graph $\Graph$ of $G$ with respect to $S$ is given by $\mathit{V\Graph}=G$ and $\mathit{E\Graph}=\{\{g,gs\} \setsep g\in G, s\in S\}$. If we define a Cayley graph in this way, namely by right multiplication, then $G$ acts as a vertex-transitive group of automorphisms on $\Graph$ by left multiplication. Hence Cayley graphs of finitely generated groups give rise to a particular class of locally finite, vertex-transitive graphs.

Since we have defined our notion of boundary for metric spaces in general, it is natural to consider $\Linear G$ and $\Project G$ not only for groups $G$ (and thus for their Cayley graphs), but also for vertex-transitive graphs in general. But as Example~\ref{example:z2boundary} shows, even for simple structures, as Cayley graphs of $\Z^d$, for Cayley graphs of groups $G$ the space $\Quotient\Unbdd$ is much richer than $\Linear G$ or $\Project G$. Hence it seems rather difficult to characterize our boundaries for graphs without involving group actions. Therefore, we define---roughly speaking---the projective (linear) boundary of a graph as the projective (linear) boundary induced by the action of its automorphism group. Then, at least for graphs with polynomial growth, it is possible to obtain results similar to the above.

Furthermore, we emphasize that the concepts defined in the sequel are not restricted to locally finite graphs. In addition the results up to Corollary~\ref{corollary:graphs} also hold without the assumption of local finiteness. From Theorem~\ref{theorem:graphs} to the end of this section we consider graphs with polynomial growth which of course implies that they are locally finite. Hence, although the main results of this section only hold for locally finite graphs, this assumption is never explicitly stated.

In the following we always endow a graph $\Graph$ with the graph metric $d$, i.e., for any two vertices $u,v\in\mathit{V\Graph}$, the distance $d(u,v)$ is the infimum of all numbers $k$ such that there is a path of length $k$ connecting $u$ and $v$.

\begin{definition}
Let $\Graph=(\mathit{V\Graph},\mathit{E\Graph})$ be an infinite, connected graph and let $\Aut\Graph$ be the automorphism group of $\Graph$. For $v\in\mathit{V\Graph}$, we write $\UnbddOrb_v\Graph\subseteq\Aut\Graph$ to denote the set of group elements $g\in\Aut\Graph$ for which the set $\gen g v = \{g^n v \setsep n\in\Z\}$ is unbounded.
\end{definition}

\begin{lemma}\label{lemma:unbddorb}
Let $\Graph$ be an infinite, connected graph, $v\in\mathit{V\Graph}$, and let $g\in\UnbddOrb_v\Graph$.
\begin{compactitem}
\item The set $\UnbddOrb_v\Graph$ is symmetric and both,
	$g^\infty v = \{g^n v \setsep n\in\N_0\}$ and
	$(g^{-1})^\infty v = \{g^{-n} v \setsep n\in\N_0\}$, are unbounded.
\item If $n$ is a nonzero integer then $g^n\in\UnbddOrb_v\Graph$.
	Furthermore, $\gen g v$, $\gen{g^n} v$ are linearly equivalent and
	$g^\infty v$, $(g^n)^\infty$ are linearly equivalent, too.
\end{compactitem}
\end{lemma}

\begin{proof}
The first part is immediate. The second part can be proved in the same way as Lemma~\ref{lemma:orbits}.
\end{proof}

\begin{definition}
Let $\Graph$ be an infinite, connected graph and let $G\le\Aut\Graph$. For $v\in \mathit{V\Graph}$ we define
\[ \Orbit_{G,v}\Graph = \{ \gen g v \setsep g \in \UnbddOrb_v\Graph \cap G \}, \qquad
	\Forward_{G,v}\Graph = \{ g^\infty v \setsep g \in \UnbddOrb_v\Graph \cap G \}. \]
\end{definition}

\begin{lemma}\label{lemma:ref-of-graph}
Let $\Graph$ be an infinite, connected graph. Then, for $u,v\in\mathit{V\Graph}$, we have
\[ \UnbddOrb_u\Graph = \UnbddOrb_v\Graph. \]
If $G\le\Aut\Graph$ then
\[ \Quotient{\Orbit_{G,u}\Graph} = \Quotient{\Orbit_{G,v}\Graph} \qquad\text{and}\qquad
	\Quotient{\Forward_{G,u}\Graph} = \Quotient{\Forward_{G,v}\Graph} \]
up to isometric isomorphy.
\end{lemma}

\begin{proof}
Assume that $g\in\UnbddOrb_u\Graph$. Then $d(u, g^n u)\to\infty$ as $n\to\infty$. As $\Graph$ is connected, $d(u,v)<\infty$ for all $v\in VX$. The triangle inequality implies
\[ d(u, g^n u) \le d(u,v) + d(v, g^n v) + d(g^n v, g^n u) = 2d(u,v) + d(v, g^n v), \]
hence $d(v, g^n v) \ge d(u, g^n u) - 2d(u,v) \to \infty$ as $n\to\infty$. Thus $\UnbddOrb_u\Graph\subseteq\UnbddOrb_v\Graph$ and the reversed inclusion follows by means of symmetry. In order to prove the second part of our assertion, note that, for $g\in\UnbddOrb_u\Graph\cap G=\UnbddOrb_v\Graph\cap G$,
\[ \gen g u \subseteq 0\gen g v + d(u,v) \qquad\text{and}\qquad
	\gen g v \subseteq 0\gen g u + d(u,v) \]
which means that $\gen g u$ and $\gen g v$ are linearly equivalent, thus implying $\Quotient{\Orbit_{G,u}\Graph} = \Quotient{\Orbit_{G,v}\Graph}$ up to isometric isomorphy. An analogous reasoning yields $\Quotient{\Forward_{G,u}\Graph} = \Quotient{\Forward_{G,v}\Graph}$.
\end{proof}

In the light of Lemma~\ref{lemma:ref-of-graph} we may drop dependence on the vertex $v$. This motivates the following definition:

\begin{definition}\label{definition:graph-bndry}
Let $\Graph$ be an infinite, connected graph and fix a reference vertex $v$. Then we define $\UnbddOrb\Graph=\UnbddOrb_v\Graph$. If $G$ is a subgroup of $\Aut\Graph$ then we set
\[ \Project_G\Graph=\Bndry{\Orbit_{G,v}\Graph} \qquad\text{and}\qquad
	\Linear_G\Graph=\Bndry{\Forward_{G,v}\Graph}. \]
The spaces $\Project\Graph=\Project_{\Aut\Graph}\Graph$ and $\Linear\Graph=\Linear_{\Aut\Graph}\Graph$ are called \emph{projective boundary} and \emph{linear boundary} of $\Graph$, respectively.
\end{definition}

Let $\Graph$ be a graph and let $\sigma$ be a partition of the vertex set $\mathit{V\Graph}$. The quotient graph $\Graph_\sigma$ is defined as follows: the vertex set $\mathit{V\Graph}_\sigma$ is $\sigma$, and two vertices $x,y\in\mathit{V\Graph}_\sigma$ are adjacent, if there are adjacent vertices $v,w\in\Graph$ with $v\in x$ and $w\in y$. Let $G\le\Aut\Graph$ be a group of automorphisms such that $\sigma$ is $G$\ndash invariant, i.e.\ $g(b)\in\sigma$ for all $b\in\sigma$ and all $g\in G$. Then $G$ naturally induces a group action on $\Graph_\sigma$. The subgroup of the automorphism group $\Aut\Graph_\sigma$ corresponding to this action is denoted by $G_\sigma$. Also, there is a homomorphism $\phi\colon G\to G_\sigma$ such that the kernel of $\phi$ consists of all those $g\in G$ with $g(b)=b$ for all $b\in\sigma$. If $G\le\Aut\Graph$ acts vertex-transitively on $\Graph$ and $\sigma$ is a $G$\ndash invariant partition of $\mathit{V\Graph}$ then $\sigma$ is called \emph{imprimitivity system of $G$ on $\Graph$}. The elements of an imprimitivity system are called \emph{blocks}.

Let $\sigma$ be an $\Aut\Graph$\ndash invariant partition of $\mathit{V\Graph}$. In order to avoid ambiguity we write $(\Aut\Graph)_\sigma$ to denote the subgroup of $\Aut\Graph_\sigma$ corresponding to the natural action of $\Aut\Graph$ on $\Graph_\sigma$. Notice that $(\Aut\Graph)_\sigma\subseteq\Aut\Graph_\sigma$, but these two groups are not necessarily equal as the next example shows.

\begin{example}\label{example:crossedlatter}
Consider the graph $\Graph$ depicted in Figure~\ref{fig:crossedlatter}. It consists of two disjoint infinite double-rays $\{v_i \setsep i\in\Z\}$ and $\{w_i \setsep i\in\Z\}$ and additional ``crossed rungs'': For even $i$, $v_i$ is connected to $w_{i+1}$ and for odd $i$, $v_i$ is connected to $w_{i-1}$.
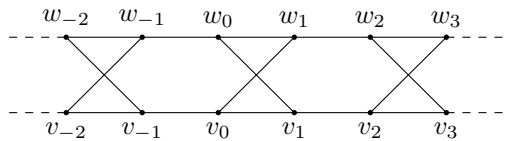
\begin{figure}[htb]
\centering
\tikzstyle{vertex}=[circle, fill=black, inner sep=0pt, minimum width=2pt]
\begin{tikzpicture}
	\draw (-2,0)--(3,0) (-2,1)--(3,1);
	\draw[dashed] (-2.75,0)--(-2,0) (-2.75,1)--(-2,1) (3,0)--(3.75,0) (3,1)--(3.75,1);
	\foreach \i in {-2,-1,0,1,2,3} {
		\ifodd\i\relax \draw (\i,0)--(\i-1,1); \else \draw (\i,0)--(\i+1,1); \fi
		\draw (\i,0) node[vertex] {} node[below,font=\small] {$v_{\i}$};
		\draw (\i,1) node[vertex] {} node[above,font=\small] {$w_{\i}$};
	}
\end{tikzpicture}
\caption{An example graph $\Graph$ for $(\Aut\Graph)_\sigma\subsetneq\Aut\Graph_\sigma$.}
\label{fig:crossedlatter}
\end{figure}
This graph is vertex-transitive, and the sets $\{v_i,w_i\}$, $i\in\Z$, give rise to an imprimitivity system $\sigma$ of $\Aut\Graph$ on $\Graph$. The quotient graph $\Graph_\sigma$ is an infinite double-ray $\{x_i \setsep i\in\Z\}$, where the vertices $x_i$ correspond to the sets $\{v_i,w_i\}$ for $i\in\Z$. The mapping $g_\sigma$ which fixes $x_0$ and maps $x_i$ onto $x_{-i}$ for $i\in\Z$ is obviously an automorphism of $\Graph_\sigma$. But there exists no automorphism $g\in\Aut\Graph$ with
\[ g(\{v_i,w_i\}) = \{v_{-i},w_{-i}\} \]
for $i\in\Z$. Hence, for this graph $\Graph$, $(\Aut\Graph)_\sigma\subsetneq\Aut\Graph_\sigma$ holds.
\end{example}

Let $\Graph$ be an infinite, connected graph and $H\le G\le\Aut\Graph$. As the underlying metric space $(\mathit{V\Graph},d)$ is fixed, the inclusion $H\le G$ implies, that $\Linear_H\Graph$ and $\Project_H\Graph$ are up to isometric isomorphy subspaces of $\Linear_G\Graph$ and $\Project_G\Graph$, respectively: Fix some reference vertex $v\in\mathit{V\Graph}$ and notice that $\Forward_{H,v}\subseteq\Forward_{G,v}$. Hence the map
\[ \Forward_{H,v}\to\Forward_{G,v}, \quad h^\infty v\mapsto h^\infty v \]
induces an isometric embedding $\Quotient{\Forward_{H,v}}\to\Quotient{\Forward_{G,v}}$ which extends naturally to the topological closures $\Linear_H\Graph$ and $\Linear_G\Graph$. Similarly, there is an isometric embedding $\Project_H\Graph\to\Project_G\Graph$.

\begin{lemma}\label{lemma:graphs}
Let $\Graph$ be an infinite, connected graph.
\begin{compactitem}
\item If $H\le G\le\Aut\Graph$ and $H$ has finite index in $G$
	then $\Linear_H\Graph$ and $\Linear_G\Graph$
	($\Project_H\Graph$ and $\Project_G\Graph$)
	are isometrically isomorphic.
\item Let $G\le\Aut\Graph$ and let $\sigma$ be a $G$\ndash invariant partition of $\mathit{V\Graph}$
	such that
	\[ \sup\{d(x,y) \setsep x,y\in b, \, b\in\sigma\}<\infty. \]
	Then $\Linear_G\Graph$ and $\Linear_{G_\sigma}\Graph_\sigma$
	($\Project_G\Graph$ and $\Project_{G_\sigma}\Graph_\sigma$)
	are bi-Lipschitz-equivalent.
\end{compactitem}
\end{lemma}

\begin{proof}
In order to prove the first statement, we may assume that $H$ is a normal subgroup of $G$ with finite index, as the intersection of all conjugates of $H$ forms a normal subgroup with finite index. Let $n$ be the finite index of $H$ in $G$. Then, for any $g\in\UnbddOrb_v\Graph\cap G$, $g^n\in\UnbddOrb_v\Graph\cap H$ and the unbounded subsets $g^\infty v\in\Forward_{G,v}$, $(g^n)^\infty v\in\Forward_{H,v}$ are linearly equivalent. Therefore, the isometric embedding $\Quotient{\Forward_{H,v}}\to\Quotient{\Forward_{G,v}}$ is an isometric isomorphism which extends naturally to $\Linear_H\Graph$ and $\Linear_G\Graph$. Analogous reasoning yields the statement for $\Project_H\Graph$ and $\Project_G\Graph$.

We now prove the second assertion. For $x\in\mathit{V\Graph}$ we write $\bar x$ to denote the element of $\sigma=\mathit{V\Graph}_\sigma$ containing $x$. Similarly, we write $\bar g\in G_\sigma$ for the automorphism of $\Graph_\sigma$ induced by the group element $g\in G$. Fix some reference vertex $v$ and set
\[ a=\sup\{d(x,y) \setsep x,y\in b, \, b\in\sigma\} < \infty. \]
The map $\pi\colon\mathit{V\Graph}\to\mathit{V\Graph}_\sigma$, $x\mapsto\bar x$, is a quasi-isometry, since
\[ d_{\Graph_\sigma}(\bar x, \bar y) \le d_\Graph(x,y) \le (a+1) d_{\Graph_\sigma}(\bar x, \bar y) + a \]
for $x,y\in\mathit{V\Graph}$. Furthermore, $\pi$ induces a map from $\Forward_{G,v}\Graph$ onto $\Forward_{G_\sigma,\bar v}\Graph_\sigma$: if $g^\infty v = \{v_0,v_1,\dotsc\} \in \Forward_{G,v}\Graph$ then $\pi(g^\infty v) = \{\bar v_0, \bar v_1, \dotsc\} = \bar g^\infty \bar v \in \Forward_{G_\sigma,\bar v}\Graph_\sigma$. Theorem~\ref{theorem:qi-bilip} implies that
\[ \Linear_G\Graph = \Bndry{\Forward_{G,v}\Graph} \qquad\text{and}\qquad
	\Linear_{G_\sigma}\Graph_\sigma = \Bndry{\Forward_{G_\sigma,\bar v}\Graph_\sigma} \]
are bi-Lipschitz-equivalent. Again the statement for $\Project_G\Graph$ and $\Project_{G_\sigma}\Graph_\sigma$ follows along the same lines.
\end{proof}

\begin{corollary}\label{corollary:graphs}
Let $\Graph$ be an infinite, connected graph.
\begin{compactitem}
\item If $G\le\Aut\Graph$ acts vertex-transitively on $\Graph$ and
	$\sigma$ is an imprimitivity system of $G$ on $\Graph$ with finite blocks
	then $\Linear_G\Graph$ and $\Linear_{G_\sigma}\Graph_\sigma$
	($\Project_G\Graph$ and $\Project_{G_\sigma}\Graph_\sigma$)
	are bi-Lipschitz-equivalent.
\item If $G\le\Aut\Graph$ acts freely and with finitely many orbits on $\mathit{V\Graph}$
	then $\Linear G$ and $\Linear_G\Graph$
	($\Project G$ and $\Project_G\Graph$)
	are bi-Lipschitz-equivalent.
\end{compactitem}
\end{corollary}

\begin{proof}
The first statement is immediate:
\[ \sup\{ d(x,y) \setsep x,y\in b, \, b\in\sigma\} < \infty \]
follows from the fact that $G$ acts vertex-transitively on $\Graph$ and the blocks of $\sigma$ are finite.

To prove the second statement we apply the ideas of the proof of the so-called Contraction Lemma (see \cite{babai1977some}): Since $G$ acts freely and with finitely many orbits on $\Graph$, there is a finite tree $T$ in $\Graph$ which contains exactly one vertex of each orbit of $G$ on $\Graph$. Furthermore, the sets $g\mathit{VT}$ for $g\in G$ form a partition of $\mathit{V\Graph}$. Set $\sigma=\{g\mathit{VT} \setsep g\in G\}$. Then $\Graph_\sigma$ is isomorphic to a Cayley graph of $G$, and the groups $G$ and $G_\sigma$ are isomorphic, as $\mathit{VT}$ contains exactly one vertex of each orbit. Hence $\Linear G$ is (by definition) equal to $\Linear_{G_\sigma}\Graph_\sigma$ and the spaces $\Linear_{G_\sigma}\Graph_\sigma$, $\Linear_G\Graph$ are bi-Lipschitz-equivalent by the previous lemma.
\end{proof}

\begin{theorem}\label{theorem:graphs}
Let $\Graph$ be an infinite, connected, vertex-transitive graph with polynomial growth. Then there is a finitely generated, torsion-free, nilpotent group $N$ which has the same growth rate as $\Graph$, and $\Linear N$ and $\Project N$ are bi-Lipschitz-equivalent to $\Linear\Graph$ and $\Project\Graph$, respectively.
\end{theorem}

To prove this result about graphs with polynomial growth, the following two results of Trofimov~\cite{trofimov1984graphs} are essential.

\begin{theorem}[Theorem~1 in \cite{trofimov1984graphs}]\label{theorem:trofimov1}
Let $\Graph$ be an infinite, connected, vertex-transitive graph with polynomial growth. Then there exists an imprimitivity system $\sigma$ of $\Aut\Graph$ on $\mathit{V\Graph}$ with finite blocks such that $\Aut\Graph_\sigma$ is a finitely generated virtually nilpotent group and the stabilizer in $\Aut\Graph_\sigma$ of a vertex of $\Graph_\sigma$ is finite.
\end{theorem}

\begin{theorem}[Theorem~2 in \cite{trofimov1984graphs}]\label{theorem:trofimov2}
Let $\Graph$ be an infinite, connected graph with polynomial growth and let a group $G\leq\Aut\Graph$ act vertex-transitively on $\mathit{V\Graph}$. Then there exists an imprimitivity system $\sigma$ of $G$ on $\mathit{V\Graph}$ with finite blocks such that $G_\sigma$ is a finitely generated virtually nilpotent group and the stabilizer in $G_\sigma$ of a vertex of $\Graph_\sigma$ is finite.
\end{theorem}

\begin{proof}[Proof of Theorem~\ref{theorem:graphs}]
Let $G=\Aut\Graph$ and let $\sigma$ and $G_\sigma$ as in Theorem~\ref{theorem:trofimov2}. Then $G_\sigma$ contains a finitely generated, nilpotent, normal subgroup $N$ of finite index. By \cite[Corollary~2.7]{seifter1991properties} we can furthermore assume that $N$ is torsion-free. Since the finite index of $N$ in $G_\sigma$ implies that $N$ acts with finitely many orbits on $\Graph$, we can assume by \cite[Theorem~2.3]{seifter1991groups} that all $n\in N$, $n\ne 1$, act with infinite orbits on $\Graph_\sigma$.

Since the vertex stabilizers of $\Aut\Graph_\sigma$ and $G_\sigma$ are both finite (by Theorems~\ref{theorem:trofimov1} and \ref{theorem:trofimov2}), both groups have the same growth rate as the graph $\Graph_\sigma$ which is of course equal to the growth rate of $\Graph$. Hence $G_\sigma$ has finite index in $\Aut\Graph_\sigma$. As $N$ has finite index in $G_\sigma$, it has also finite index in $\Aut\Graph_\sigma$. Therefore Lemma~\ref{lemma:graphs} implies that the projective (linear) boundary induced by $N$ on $\Graph_\sigma$ is bi-Lipschitz-equivalent to the projective (linear) boundary induced by $\Aut\Graph_\sigma$ which we defined to be the projective (linear) boundary of $\Graph_\sigma$.

Since $\Graph_\sigma$ is a quotient graph of $\Graph$ with respect to the finite blocks of $\sigma$, Corollary~\ref{corollary:graphs} implies that the projective (linear) boundaries of $\Graph$ and $\Graph_\sigma$ which are induced by $\Aut\Graph$ and $G_\sigma=(\Aut\Graph)_\sigma$, respectively, are bi-Lipschitz-equivalent.

To conclude the proof we show that $\Linear N$ and $\Project N$ are bi-Lipschitz-equivalent to $\Linear_N\Graph_\sigma$ and $\Project_N\Graph_\sigma$, respectively. As $N$ is torsion-free and the stabilizer of a vertex is finite, $N$ acts freely on $\Graph_\sigma$. Since $N$ also acts with finitely many orbits on $\Graph_\sigma$, the claim follows directly from Corollary~~\ref{corollary:graphs}.
\end{proof}

As a consequence of Theorem~\ref{theorem:nilpotent} we obtain the following result.

\begin{corollary}\label{corollary:grdis}
Let $\Graph$ be an infinite, connected, vertex-transitive graph with polynomial growth and let $N$ be a finitely generated, torsion-free, nilpotent group supplied by Theorem~\ref{theorem:graphs}. Then the linear boundary $\Linear\Graph$ is homeomorphic to a disjoint union of spheres:
\[ \Linear\Graph = \S^{\nu(1)-1} \uplus \S^{\nu(2)-1}\uplus \dotsb \uplus \S^{\nu(c)-1}, \]
where $c$ is the nilpotency class of $N$ and $\nu(i)$ is the torsion-free rank of the $i$\ndash th quotient in the descending central series of $N$. Analogously, the projective boundary $\Project\Graph$ is homeomorphic to a disjoint union of projective spaces:
\[ \Project\Graph = \P^{\nu(1)-1} \uplus \P^{\nu(2)-1}\uplus \dotsb \uplus \P^{\nu(c)-1}. \]
\end{corollary}

Having these characterizations of the linear and projective boundaries of vertex-transitive graphs with polynomial growth, immediately the following question arises: When are the linear (projective) boundary of an infinite, connected, vertex-transitive graph $\Graph$ with polynomial growth and the linear (projective) boundary of its automorphism group $\Aut\Graph$ bi-Lipschitz-equivalent? Using the concept of bounded automorphisms we are able to present a partial answer to this question.

An automorphism $b\in\Aut\Graph$ is called \emph{bounded} if there is an integer $k$, depending on $b$, such that $d(x,b(x))\leq k$ holds for all $x\in\mathit{V\Graph}$. Of course the bounded automorphisms of $\Graph$ give rise to a normal subgroup $B(\Graph)$ of $\Aut\Graph$. As was shown in \cite{godsil1989note}, the same holds for the bounded automorphisms of finite order of $\Graph$. We denote the normal subgroup of $\Aut\Graph$ generated by all bounded automorphisms of finite order by $B_0(\Graph)$. As was also shown in \cite{godsil1989note}, $B_0(\Graph)$ is locally finite, periodic and has finite orbits on $\Graph$. Furthermore, in \cite{seifter1991properties} the following result concerning $B_0(\Graph)$ was proved:

\begin{proposition}[Corollary~2.7 in \cite{seifter1991properties}]
\label{proposition:bounded-finite}
Let $\Graph$ be an infinite, connected graph with polynomial growth and let $G\le \Aut\Graph$ act vertex-transitively on $\Graph$. Then the orbits of $B_0(\Graph)\cap G$ on $\Graph$ give rise to an imprimitivity system $\sigma$ of $G$ on $\mathit{V\Graph}$ such that $G_{\sigma}$ satisfies the assertions of Theorem~\ref{theorem:trofimov2}.
\end{proposition}

Together with the following result of Sabidussi~\cite{sabidussi1964vertex}, Proposition~\ref{proposition:bounded-finite} now immediately implies a partial answer to the above formulated question. To formulate Sabidussi's result we need another definition.

If $\Graph$ is a graph and $m$ is a cardinal then the graph $m\Graph$ is defined on the Cartesian product of $\mathit{V\Graph}$ by a set $M$ of cardinality $m$, and
\[ E(m\Graph) = \Bigl\{ \{(x,i),(y,j)\} \setsep \{x,y\}\in \mathit{E\Graph}, \, i,j\in M \Bigr\}. \]

\begin{theorem}[Theorem~4 in \cite{sabidussi1964vertex}]
\label{theorem:sabidussi}
Let $\Graph$ be a connected graph and let $G\le \Aut\Graph$ act vertex-transitively on $\Graph$. Furthermore, let $m$ denote the cardinality of the stabilizer in $G$ of a vertex of $\Graph$. Then $m\Graph$ is a Cayley graph of $G$.
\end{theorem}

\begin{corollary}
Let $\Graph$ be an infinite, connected, vertex-transitive graph with polynomial growth. Then $\Linear\Aut\Graph$ and $\Project\Aut\Graph$ are bi-Lipschitz-equivalent to $\Linear\Graph$ and $\Project\Graph$, respectively, if $B_0(\Graph)$ is finite.
\end{corollary}

\begin{proof}
$B_0(\Graph)$ is a normal subgroup of $\Aut\Graph$. If it is in addition finite then it follows from \ref{proposition:bounded-finite} and \ref{theorem:trofimov1} that the stabilizer of a vertex of $\Graph$ in $\Aut\Graph$ has some finite cardinality $m$. Then, by Theorem~\ref{theorem:sabidussi}, $m\Graph$ is a Cayley graph of $\Aut\Graph$ and arguments quite similar to those in the proof of Theorem~\ref{theorem:graphs} immediately complete the proof.
\end{proof}

In \cite{trofimov1983bounded} Trofimov defined a \emph{lattice} as a connected locally finite graph $\Graph$, such that for one of the groups $G$, acting vertex-transitively on $\Graph$, there exists an imprimitivity system $\sigma$ with finite blocks, such that $G_\sigma$ is a finitely generated, commutative group. As was shown in \cite{trofimov1983bounded}, in this case $G\leq B(\Graph)$ holds. Furthermore, it is obvious that lattices have polynomial growth with the same growth rate as $G_\sigma$. In addition lattices can be characterized as follows:

\begin{theorem}[Theorem~1 in \cite{trofimov1983bounded}]
\label{theorem:lattices}
Let $\Graph$ be a connected locally finite graph. Then $\Graph$ is a lattice if and only if a group $G\leq B(\Graph)$ acts vertex-transitively on $\Graph$.
\end{theorem}

This immediately leads to the following:

\begin{theorem}\label{theorem:latticebndry}
Let $\Graph$ be a connected locally finite graph of polynomial growth with growth rate $r$ and let a group $G\leq B(\Graph)$ act vertex-transitively on $\Graph$. Then
\[ \Linear\Graph = \S^{r-1} \qquad\text{and}\qquad \Project\Graph = \P^{r-1}. \]
\end{theorem}

\begin{proof}
Applying Theorem~\ref{theorem:lattices} this result can be shown analogously to the proof of Theorem~\ref{theorem:graphs}.
\end{proof}

Let $\Graph$ now be a Cayley graph of a group $G$. Then any group element $g\in G$ gives rise to a bounded automorphism of $\Graph$ if and only if the conjugacy class of $g$ in $G$ is finite (see e.g.\ \cite[page~335]{godsil1989note}). So the boundedness of an element $g\in G$ is independent of whatever Cayley graph represents $G$.

A group $G$ is called \emph{$FC$\ndash group} if for every $g\in G$ the conjugacy class of $g$ in $G$ is finite. Hence for $FC$\ndash groups $G$ each $g\in G$ acts as a bounded automorphism on any Cayley graph of $G$. Therefore Cayley graphs of finitely generated $FC$\ndash groups are lattices and Theorem~\ref{theorem:latticebndry} immediately implies:

\begin{corollary}
Let $G$ be a finitely generated $FC$\ndash group with polynomial growth of growth rate $r$. Then
\[ \Linear G = \S^{r-1} \qquad\text{and}\qquad \Project G = \P^{r-1}. \]
\end{corollary}

\section{Attaching the boundary}
\label{section:attach}

Let $\Xi$ be any subset of $\Quotient\Unbdd$. In the following we describe a topology $\tau$ on the disjoint union $\bar X$ of $X$ and $\Xi$, such that two requirements hold:
\begin{compactitem}
\item The subspace topology of $\tau$ on $X$ is induced by the metric $d$.
\item If $x_1,x_2,\dotsc$ is a sequence in $X$,
	which eventually leaves any ball in $X$,
	and $\xi$ is an equivalence class in $\Xi$,
	such that $x_1,x_2,\dotsc\in R$ for some $R\in\xi$
	then $x_1,x_2,\dotsc$ converges to $\xi$ in $\tau$.
\end{compactitem}
Due to the second requirement the subspace topology of $\tau$ on $\Xi$ is in general neither induced by the metric $t$ nor Hausdorff, see Lemma~\ref{lemma:hausdorff}.

Fix some reference point $o$ in $X$ and let $\xi\in\Xi$ be an equivalence class. If $R\in\xi$ and $\alpha>0$ and $r\ge0$ then we set
\[ N(R,\alpha,r) = \interior\bigl(\alpha R \setminus \OpenBall(o,r)\bigr)
	\uplus \{ \zeta\in\Xi \setsep s^+(\xi,\zeta) < \alpha \} \]
where $\interior(A)$ is the interior of the set $A\subseteq X$. Note that $N(R,\alpha,p)\subseteq N(S,\beta,q)$ if $R\subseteq S$, $\alpha\le\beta$, $p\ge q$. We define the topology $\tau$ on $\bar X = X\uplus\Xi$ by assigning to each $x\in\bar X$ a family $\mathcal V_x$ of sets which serves as an open neighborhood base for $x$:
\begin{compactitem}
\item If $x\in X$ then $\mathcal V_x$ is the family of open balls centered at $x$.
\item If $\xi\in\Xi$ then $\mathcal V_\xi$ is the family of sets $N(R,\alpha,r)$
	with $R\in\xi$, $\alpha>0$, and $r\ge0$.
\end{compactitem}

\begin{lemma}\label{lemma:gluing}
The families $\mathcal V_x$, $x\in\bar X$, are open neighborhood bases of a topology $\tau$ on $\bar X$. Its subspace topology on $X$ is induced by the metric $d$, $X$ is dense and open in $\bar X$, and the subspace topology on $\Xi$ is $T_0$.
\end{lemma}

\begin{proof}
By Theorem~4.5 in \cite{willard2004general} we have to check the following three conditions for all $x\in\bar X$:
\begin{compactitem}
\item If $V\in\mathcal V_x$ then $x\in V$.
\item If $V_1,V_2\in\mathcal V_x$ then $V_3\subseteq V_1\cap V_2$ for some $V_3\in\mathcal V_x$.
\item If $V\in\mathcal V_x$ and $z\in V$ then $W\subseteq V$ for some $W\in\mathcal V_z$.
\end{compactitem}
The first condition is immediate for all $x\in\bar X$ and the second and third condition hold for all $x\in X$. Hence let $\xi\in\Xi$. In order to prove the second condition for $\xi$ consider $N(R,\alpha,p), N(S,\beta,q) \in \mathcal V_\xi$ with $R,S\in\xi$, $\alpha,\beta>0$, and $p,q\ge0$. Choose $\eps$ in $(0,\beta)$ and set
\[ \gamma = \min\bigl\{\alpha,\tfrac{\beta-\eps}{1+\eps}\bigr\}. \]
Since $R,S\in\xi$, it follows that $s(R,S)=0$ and by Lemma~\ref{lemma:alternative} there is a number $r\ge\max\{p,q\}$ such that $R\setminus\OpenBall(o,r)\subseteq\eps S$. Using Lemma~\ref{lemma:transitive} this yields
\[ \gamma R \subseteq \gamma(R\setminus\OpenBall(o,r)) \cup \gamma\OpenBall(o,r)
	\subseteq (\gamma+\eps\gamma+\eps) S \cup \OpenBall(o,(1+\gamma)r)
	\subseteq \beta S \cup \OpenBall(o,(1+\gamma)r) \]
by the choice of $\gamma$. Therefore
\[ N(R,\gamma,(1+\gamma)r) \subseteq N(R,\alpha,p) \cap N(S,\beta,q), \]
whence the second condition holds for $\xi$. The third condition holds for $\xi$, if $z\in V\cap X$ or $z=\xi$. Hence consider $V=N(R,\alpha,p)$ with $R\in\xi$, $\alpha>0$, $p\ge0$, and let $\zeta\ne\xi$ be an element in $V\cap\Xi$. Choose an element $S$ in $\zeta$ and choose $\beta$ in $(s^+(R,S),\alpha)$, which is possible, since $s^+(R,S) = s^+(\xi,\zeta) < \alpha$. There is a number $r\ge p$, such that $S\setminus\OpenBall(o,r) \subseteq \beta R$. Set $\gamma = \frac{\alpha-\beta}{1+\beta} > 0$. Then
\[ \gamma S \subseteq \gamma(S\setminus\OpenBall(o,r)) \cup \gamma\OpenBall(o,r)
	\subseteq (\gamma+\beta\gamma+\beta) R \cup \OpenBall(o,(1+\gamma)r)
	= \alpha R \cup \OpenBall(o,(1+\gamma)r) \]
by the choice of $\beta$ and $\gamma$. Hence we obtain
\[ N(S,\gamma,(1+\gamma)r) \subseteq N(R,\alpha,p). \]
The last three assertions follow from the construction of $\tau$.
\end{proof}

\begin{remark}
Let $X$ be an unbounded, locally compact, metric space. Then $(\bar X,\tau)$ is compact if the equivalence class of the unbounded set $X$ is an element of $\Xi$. If, apart from the equivalence class of $X$, $\Xi$ contains further elements then $(\bar X,\tau)$ is not Hausdorff.
\end{remark}

\begin{lemma}\label{lemma:hausdorff}
Let $\Xi$ be any subset of $\Quotient\Unbdd$ and let $(\bar X,\tau)$ be defined as above.
\begin{compactitem}
\item The space $(\bar X,\tau)$ is Hausdorff if and only if
	\[ s^+(\xi,\zeta) = 0 \Longleftrightarrow s^+(\zeta,\xi) = 0 \]
	for all $\xi,\zeta\in\Xi$.
	In this case, the subspace topology of $\tau$ on $\Xi$ is induced by the metric $t$.
\item Suppose that $\Xi=\Bndry\Elements$ for some family $\Elements\subseteq\Unbdd$.
	If there exists a function $f\colon[0,1]\to[0,\infty)$,
	such that $f(0)=0$, $f$ is continuous at $0$, and
	$s^+(S,R)\le f(s^+(R,S))$ for all $R,S\in\Elements$
	then $(\bar X,\tau)$ is Hausdorff and
	the subspace topology of $\tau$ on $\Bndry\Elements$ is induced by the metric $t$.
\end{compactitem}
\end{lemma}

\begin{proof}
The first assertion is a direct consequence of the definition of the open neighborhood bases $\mathcal V_\xi$ for $\xi\in\Xi$. The second statement is a consequence of the first, since the hypotheses imply that
\[ s^+(\xi,\zeta) = 0 \Longleftrightarrow s^+(\zeta,\xi) = 0 \]
for all $\xi,\zeta\in\Bndry\Elements$: If $s^+(\xi,\zeta)=0$ and $\eps>0$ is given then there are $\xi',\zeta'\in\Quotient\Elements$, such that $s(\xi,\xi')\le\eps$ and $s(\zeta,\zeta')\le\eps$. Thus
\begin{align*}
s^+(\zeta,\xi)
&\le 2\eps + \eps^2 + s^+(\zeta',\xi') (1+\eps)^2 \\
&\le 2\eps + \eps^2 + f(s^+(\xi',\zeta')) (1+\eps)^2 \\
&\le 2\eps + \eps^2 + f(2\eps + \eps^2) (1+\eps)^2.
\end{align*}
This shows that $s^+(\zeta,\xi)=0$.
\end{proof}

With these preparations we are able to provide a criterion which ensures that the topology defined above on the disjoint union of a compactly generated, locally compact Hausdorff group $G$ and its linear boundary $\Linear G$ (projective boundary $\Project G$) is Hausdorff and the subspace topology on $\Linear G$ ($\Project G$) is induced by the angle metric $t$.

\begin{proposition}\label{proposition:group-bnd}
Let $G$ be a compactly generated, locally compact Hausdorff group. Assume that there exists a constant $C\ge1$, such that for every group element $g\in G$ with $\genp g\in\Forward G$ there is an element $\tilde g\in G$ with the following two properties:
\begin{compactitem}
\item $\genp{\tilde g} \linequiv \genp g$ and
\item $d(1,\tilde g^m) \le C d(1,\tilde g^n) + C$ for all $m,n$ with $0\le m\le n$.
\end{compactitem}
Then the topology $\tau$ on $G\uplus\Linear G$ defined by Lemma~\ref{lemma:gluing} is Hausdorff and the subspace topology of $\tau$ on $\Linear G$ is induced by the metric $t$. An analogous statement holds for the projective boundary.
\end{proposition}

\begin{proof}
We check that the function $f\colon[0,1]\to[0,\infty)$, $x\mapsto 2(1+4C)x$ satisfies the conditions of the second part of Lemma~\ref{lemma:hausdorff} which implies the statement.

Of course, $f$ is continuous and $f(0)=0$. Furthermore, if $x\ge\tfrac12$, then $f(x)\ge 1+4C\ge 1$. Hence $s^+(\genp h,\genp g)\le f(s^+(\genp g,\genp h))$ is trivially true if $\genp g,\genp h\in\Forward G$ and $s^+(\genp g,\genp h)\ge\tfrac12$, since $s^+(\genp h,\genp g)\le1$. Hence we may assume that $s^+(\genp g,\genp h)<\tfrac12$. Additionally, after replacing $g$ by $\tilde g$ if necessary, we may assume that $d(1,g^m) \le C d(1,g^n)+ C$ for all $m,n$ with $m\le n$. Choose a number $\alpha$ which satisfies $s^+(\genp g,\genp h)<\alpha<\tfrac12$. Then there is a constant $a\ge0$, such that $\genp h\subseteq\alpha \genp g+a$. Hence, for each $n\in\N_0$ there is an integer $\nu(n)\ge0$ such that
\[ d(h^n,g^{\nu(n)}) \le \alpha d(1,g^{\nu(n)}) + a. \]
Now we define the function $\kappa\colon\N_0\to\N_0$ by
\[ \kappa(n) = \min\{m\in\N_0 \setsep \nu(m) \le n \le \nu(m+1) \}. \]
We claim that
\[ d(g^n,h^{\kappa(n)}) \le 2(1+4C) \alpha d(1,h^{\kappa(n)}) + 2Cd(1,h) + 2a + 8Ca + C \]
for all $n\in\N_0$. Once this claim is established then, by the second assertion of Lemma~\ref{lemma:hausdorff}, the proof is finished. Let $n\ge0$ be an integer and set $k=\kappa(n)$. Since
\[ d(g^n,h^k) \le d(g^n,g^{\nu(k)}) + d(g^{\nu(k)},h^k), \]
we need to find upper bounds for $d(g^n,g^{\nu(k)})$ and $d(g^{\nu(k)},h^k)$, see Figure~\ref{figure:triangles}.
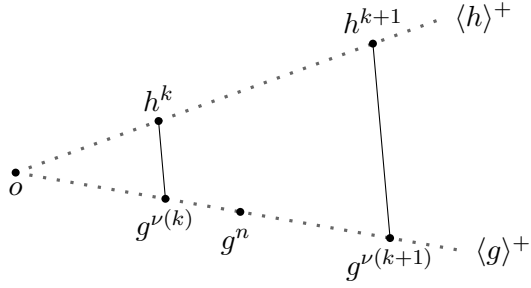
\begin{figure}[htb]
\tikzstyle{vertex}=[circle, fill=black, inner sep=0pt, minimum width=3pt]
\begin{tikzpicture}
	\draw[very thick,loosely dotted,black!60] (0,0) -- (20:6) (0,0) -- (-10:6);
	\draw (20:2) -- (-10:2) (20:5) -- (-10:5);
	\draw (0,0) node[vertex] {} node[below] {$o$};
	\draw (20:6) node[right] {$\genp h$}
	      (20:2) node[vertex] {} node[above] {$h^k$}
	      (20:5) node[vertex] {} node[above] {$h^{k+1}$}
	      (-10:6) node[right] {$\genp g$}
	      (-10:2) node[vertex] {} node[below] {$g^{\nu(k)}$}
	      (-10:3) node[vertex] {} node[below] {$g^{n\phantom{)}}$}
	      (-10:5) node[vertex] {} node[below] {$g^{\nu(k+1)}$};
\end{tikzpicture}
\caption{The positive powers of two elements $g,h$ and constellation used in the proof of Proposition~\ref{proposition:group-bnd}.}
\label{figure:triangles}
\end{figure}
Then
\[ d(1,g^{\nu(k)}) \le d(1,h^k) + d(h^k,g^{\nu(k)}) \le d(1,h^k) + \alpha d(1,g^{\nu(k)}) + a \]
yields
\[ d(1,g^{\nu(k)}) \le \tfrac1{1-\alpha} (d(1,h^k)+a) \le 2d(1,h^k) + 2a \]
using the bound $\alpha\le\tfrac12$. Thus
\[ d(g^{\nu(k)},h^k) \le \alpha d(1,g^{\nu(k)}) + a \le 2\alpha d(1,h^k) + 2a. \]
We obtain
\begin{align*}
d(g^{\nu(k)},g^{\nu(k+1)})
&\le d(g^{\nu(k)},h^k) + d(h^k,h^{k+1}) + d(h^{k+1},g^{\nu(k+1)}) \\
&\le \alpha d(1,g^{\nu(k)}) + a + d(1,h) + \alpha d(1,g^{\nu(k+1)}) + a.
\end{align*}
Then $d(1,g^{\nu(k+1)})\le d(1,g^{\nu(k)}) + d(g^{\nu(k)},g^{\nu(k+1)})$ implies
\[ d(g^{\nu(k)},g^{\nu(k+1)}) \le \alpha d(g^{\nu(k)},g^{\nu(k+1)}) + 2\alpha d(1,g^{\nu(k)}) + d(1,h) + 2a \]
and by rearranging the last inequality we get
\begin{align*}
d(g^{\nu(k)},g^{\nu(k+1)})
&\le \tfrac1{1-\alpha} (2\alpha d(1,g^{\nu(k)}) + d(1,h) + 2a) \\
&\le 4\alpha d(1,g^{\nu(k)}) + 2d(1,h) + 4a \\
&\le 8\alpha d(1,h^k) + 2d(1,h) + 8a
\end{align*}
using the bound $\alpha\le\tfrac12$ twice. The assumption on $g$ implies
\[ d(g^n,g^{\nu(k)}) \le Cd(g^{\nu(k)},g^{\nu(k+1)}) + C \le 8C\alpha d(1,h^k) + 2Cd(1,h) + 8Ca + C. \]
Collecting the pieces yields
\[ d(g^n,h^k) \le 2(1+4C) \alpha d(1,h^k) + 2Cd(1,h) + 2a + 8Ca + C. \qedhere \]
\end{proof}

\begin{lemma}\label{lemma:group-bnd}
Let $G$ be a connected, nilpotent Lie group or a finitely generated, nilpotent group. Then the assumption of the previous proposition on $G$ holds.
\end{lemma}

\begin{proof}
Without loss of generality we may assume that $G$ is simply connected in the Lie case or torsion-free in the discrete case, see Lemma~\ref{lemma:ranks} and Corollary~\ref{corollary:milnor1}. Furthermore, it is sufficient to prove the statement in the Lie case, as the discrete case follows by embedding $G$ in its real Mal'tsev completion.

Hence suppose that $G$ is a connected, simply connected, nilpotent Lie group and let $d_G$ be a word metric on $G$. We use the notation of Appendix~\ref{appendix:groups}. By Lemma~\ref{lemma:gaugequasi} there exists a constant $q$, such that
\[ q^{-1} \gauge{x} \le d_G(1,\exp(x)) \le q \gauge{x}+q \]
for all $x\in\LieG$. We claim that the assumption of the previous proposition holds for $C=q^2$. Let $g$ be a group element of $G$. Then $g\in G_i$ but $g\notin G_{i+1}$ for some $i\ge1$. Set $y=\pi_i(\log(g))\in V_i$ and $h=\exp(y)$. Then, for $0\le m\le n$, we have $\gauge{y^m} = m^{1/i} \gauge{y} \le n^{1/i} \gauge{y} = \gauge{y^n}$ and therefore
\[ d_G(1,h^m) \le q \gauge{y^m} + q \le q \gauge{y^n} + q \le q^2 d_G(1,h^n) + q. \qedhere \]
\end{proof}

\section{Random walks on nilpotent groups}
\label{section:walks}

Many aspects of random walks on nilpotent groups were studied, see for instance \cite{alexopoulos2002random,guivarch1973croissance,guivarch1980sur,kaimanovich1991poisson,tanaka2011large}. In the sequel we give a  simple corollary of some results of Kaimanovich in \cite{kaimanovich1991poisson}. Let $G$ be a connected, simply connected, nilpotent Lie group with descending central series
\[ G=G_1\supseteq G_2\supseteq \dotsc G_c\supsetneq G_{c+1}=\{1\} \]
and let $d_G$ be a word metric on $G$. A random walk $(S_k)_{k\ge0}$ on $G$ has \emph{finite first moment}, whenever
\[ E(d_G(1,S_1)) = \int_G d_G(1,g) d\mu(g) < \infty, \]
where $\mu$ is the law of $S_1$. Note that this notion does not depend on the choice of the word metric. We say that $(S_k)_{k\ge0}$ has \emph{drift} if there is an integer $n\ge1$, such that $S_1\in G_n$ almost surely and $(S_k G_{n+1})_{k\ge0}$ is a random walk in the commutative group $G_n/G_{n+1}$ with drift, i.e. if we identify $G_n/G_{n+1}$ with $\R^{\nu(n)}$, where $\nu(n)$ is the dimension of $G_n/G_{n+1}$, then the expected direction $E(S_1 G_{n+1})\in G_n/G_{n+1}=\R^{\nu(n)}$ is non-zero.

\begin{theorem}\label{theorem:walks}
Let $(S_k)_{k\ge0}$ be a random walk with finite first moment and drift on a connected, simply connected, nilpotent Lie group $G$. Then there is a deterministic group element $g$, such that $\{S_k \setsep k\ge0\}\linequiv \genp g$ holds almost surely. In terms of the topology on $G\uplus\Linear G$, of Lemma~\ref{lemma:gluing}, this means that almost surely $(S_k)_{k\ge0}$ converges to the equivalence class of $\genp g$ in $\Linear G$. On the other hand, every point in $\Linear G$ is limit point of a random walk with drift (in the sense above).
\end{theorem}

\begin{proof}
Let $n\ge1$ be the integer, such that $S_1\in G_n$ almost surely and $(S_k G_{n+1})_{k\ge0}$ is a random walk with drift. By triviality of Poisson boundary and a result of Kaimanovich (see Theorem~4.2 and the following Remark in \cite{kaimanovich1991poisson}) there is a deterministic group element $g\in G_n$ ($g\notin G_{n+1}$ by the assumptions), such that $d_{G_n}(S_k,g^k) = o(k)$ almost surely. This implies $d_G(S_k,g^k) = o(k^{1/n})$. Since $d_G(1,g^k)\ge C(g) k^{1/n}$ for some constant $C(g)>0$, we obtain $\{S_k \setsep k\ge0\}\linequiv \genp g$ almost surely. And the other statements follow.

On the other hand, every point in $\Linear G$ is limit point of a corresponding deterministic random walk.
\end{proof}

\begin{remark}
As pointed out by Tanaka \cite{tanaka2012remark} the deterministic group element $g$ in the previous theorem is given by $g G_{n+1} = E(S_1 G_{n+1})$, where $n$ is the integer, such that $S_1\in G_n$ almost surely and $(S_k G_{n+1})_{k\ge0}$ is a random walk with drift.
\end{remark}

\begin{remark}
A similar statement holds for finitely generated, torsion-free, nilpotent groups. Suppose that $G$ is such a group and consider a random walk $(S_k)_{k\ge0}$ on $G$ with drift. Then $(S_k)_{k\ge0}$ converges to an element in $\Linear G$ with respect to the topological space $(G\uplus\Linear G,\tau)$, where $\tau$ is the topology of Lemma~\ref{lemma:gluing}. On the other hand, every point in $\Linear G$ is limit point of a random walk with drift.
\end{remark}

\begin{appendix}

\section{Compactly generated groups}
\label{appendix:compact}

We provide some results on word metrics of compactly generated, locally compact groups and related issues which are completely analogous to the case of finitely generated groups. The books of Hewitt and Ross~\cite{hewitt1979abstract}, Stroppel~\cite{stroppel2006locally}, and de~la~Harpe~\cite{delaharpe2000topics} provide a good background on topological and finitely generated groups. We recall some basics from \cite{guivarch1980sur}.

\begin{lemma}[Proposition~1 in \cite{guivarch1980sur}]\label{lemma:basics}
Let $G$ be a compactly generated, locally compact Hausdorff group.
\begin{compactitem}
\item If $S$ is a compact, symmetric, generating set
	then, for some $n\ge 0$, the set $S^n$ contains a neighborhood of $1$.
\item A subset of $G$ is compact, if and only if
	it is closed and bounded with respect to some word metric.
	Consequently, a subset is bounded if and only if it is relatively compact.
\item If $S$ and $S'$ are two compact symmetric generating sets
	then the associated word metrics $d$ and $d'$ are bi-Lipschitz-equivalent,
	i.e.\ there is a constant $q>0$, such that
	\[ q^{-1} d(x,y) \le d'(x,y) \le q d(x,y) \]
	for all $x,y\in G$.
\end{compactitem}
\end{lemma}

\begin{proof}
For sake of completeness we provide a short proof: Since $G$ is Hausdorff, the sets $S^n$, $n=0,1,\dotsc,$ are closed and their union is equal to $G$. Hence, as locally compact Hausdorff spaces are Baire spaces (see \cite[Corollary~25.4]{willard2004general}), there is an integer $n\ge0$, such that $S^n$ contains a non-empty open subset. Since $S^n$ is symmetric, $S^{2n}$ contains a neighborhood of $1$.

Let $S$ be any compact, symmetric, generating subset of $G$ and $d$ be the associated word metric. Choose $n\ge0$ such that $S^n$ contains some open neighborhood $U$ of $1$. Suppose that $A$ is a compact subset of $G$. Then there are finitely many elements $a_1,\dotsc,a_r$ of $A$ such that $A\subseteq a_1 U \cup \dotsb \cup a_r U$. Thus
\[ d(1,a) \le n \max\{d(1,a_1),\dotsc,d(1,a_r)\} \]
for all $a\in A$. Hence $A$ is bounded with respect to $d$ and, since $G$ is assumed to be Hausdorff, the set $A$ is also closed. Now suppose that $A$ is a closed subset of $G$ and bounded with respect to $d$. Then $A\subseteq S^m$ for some $m\ge0$ which implies that $A$ is compact.

By the second statement, there is a constant $q>0$, such that $d'(1,x) \le q$ for all $x\in S$ and $d(1,y) \le q$ for all $y\in S'$. This implies the third assertion.
\end{proof}

A metric space $(X,d)$ is called \emph{$q$\ndash quasi-geodesic}, if for all $x,y\in X$ there is an integer $n\ge0$ and points $x=x_0,x_1,\dotsc,x_{n-1},x_n=y$ in $X$, such that
\[ n \le q d(x,y) + q \qquad\text{and}\qquad d(x_{i-1},x_i) \le q \]
for all $1\le i\le n$. We remark that similar notions are used in the literature (see for instance \cite[Definition~8.22]{bridson1999metric} and \cite[Section~0.2.D]{gromov1993asymptotic}). Of course, any geodesic metric space is $1$\ndash quasi-geodesic and any word metric on a compactly generated, locally compact group is $1$\ndash quasi-geodesic.

In the following we give a straightforward generalization of the classical Milnor-\v Svarc lemma (see for instance \cite[Proposition~8.19]{bridson1999metric} or \cite[Theorem~IV.B.23]{delaharpe2000topics}) to the continuous case. Before stating the lemma we give a precise description of the setting: Let $G$ be a locally compact group and $X$ be a Hausdorff space. Furthermore, let $d_X$ be a quasi-geodesic metric on $X$ (we do not assume that $d_X$ induces the topology on $X$). If not stated otherwise, all topological notions concerning $X$ refer to the topology on $X$ with the exception of boundedness, which refers to the metric $d_X$. An action $G\times X\to X$, $(g,x)\mapsto gx$ is called
\begin{compactitem}
\item \emph{continuous}, if it is a continuous mapping from $G\times X$ to $X$,
\item \emph{$q$\ndash cobounded}, if for all $x,y\in X$ there is a $g\in G$ with $d_X(gx,y)\le q$,
\item \emph{proper}, if $\{ g\in G \setsep d_X(gx,x)\le r\}$ is compact for all $x\in X$ and all $r\ge0$.
\end{compactitem}
We say that $G$ \emph{acts by isometries}, if $x\mapsto gx$ is an isometry with respect to $d_X$ for all $g\in G$. Note that if the action is continuous and $K\subseteq G$ is compact then, for any $x\in X$, the set $Kx = \{gx \setsep g\in K\}$ is compact and hence bounded. With these preparations we are ready to state the lemma:

\begin{lemma}\label{lemma:milnor}
Let $G$ be a locally compact Hausdorff group and $X$ be a Hausdorff space which is additionally endowed with a quasi-geodesic metric $d_X$, such that all compact subsets are bounded. Suppose that there is a continuous, cobounded, proper action of $G$ by isometries on $X$. Then $G$ is compactly generated and for any $x\in X$ the map $G\to X$, $g\mapsto gx$ is a quasi-isometry from $(G,d_G)$ to $(X,d_X)$, where $d_G$ is some word metric of $G$.
\end{lemma}

\begin{proof}
Except for minor modifications the proof is the same as in \cite{bridson1999metric,delaharpe2000topics}.

For simplicity we assume that the constant $q$ involved in the quasi-geodesic metric is the same as the constant $q$ of the cobounded action. Fix $x\in X$. Since the action is proper, the set $\{g\in G \setsep d(gx,x) \le 3q\}$ is compact. Let $S$ be the union of this set and its inverse. Then $S$ is compact and symmetric and $1\in S$.

We show that $S$ generates $G$. Let $g\in G$. Since $(X,d_X)$ is $q$\ndash quasi-geodesic, there are $x=x_0,x_1,\dotsc,x_n=gx$, such that $n\le q d(x,gx) + q$ and $d(x_{i-1},x_i)\le q$ for $1\le i\le n$. Since the action is $q$\ndash cobounded, there are group elements $g_0=1, g_1, \dotsc, g_n=g$, such that $d_X(g_ix, x_i) \le q$ for all $0\le i\le n$. Then
\[ d_X(g_{i-1}^{-1}g_i^{} x, x) = d_X(g_i x, g_{i-1} x)
	\le d_X(g_i x, x_i) + d_X(x_i, x_{i-1}) + d_X(x_{i-1}, g_{i-1} x) \le 3q. \]
It follows that $s_i^{} = g_{i-1}^{-1} g_i^{} \in S$ and thus $g=g_n=s_1\dotsm s_n\in S^n$. Hence $S$ is a generating set. Let $d_G$ be the word metric on $G$ with respect to $S$. Then the estimate above for $g\in G$ yields
\[ d_G(1,g) \le n \le q d_X(x, gx) + q. \]

Now we prove that $G\to X$, $g\mapsto gx$ is a quasi-isometry from $(G,d_G)$ to $(X,d_X)$. Let $g,h\in G$. Then we obtain
\[ d_G(g,h) = d_G(1,g^{-1}h) \le q d_X(x, g^{-1}hx) + q = q d_X(gx, hx) + q. \]
For the reversed bound, note that $Sx$ is bounded, since $S$ is compact. Hence
\[ M = \sup\{ d_X(x,y) \setsep y\in Sx\} \]
is finite. Suppose that $d_G(g,h)=n\ge1$ and $g^{-1}h=s_1\dotsm s_n$ for some $s_1,\dotsc,s_n\in S$. Then
\begin{align*}
d_X(gx, hx)
&= d_X(x, g^{-1}hx) = d_X(x, s_1\dotsm s_nx) \\
&\le d_X(x, s_1x) + d_X(s_1x, s_1s_2x) + \dotsb + d_X(s_1\dotsm s_{n-1}x, s_1\dotsm s_nx) \\
&= d_X(x, s_1x) + d_X(x, s_2x) + \dotsb + d_X(x, s_nx) \\
&\le Mn = M d_G(g,h). \qedhere
\end{align*}
\end{proof}

In order to have a handy reference we formulate the following well-known results, see \cite[Section~5]{hewitt1979abstract} and \cite[Section~I.10.2]{bourbaki1966general1}.

\begin{lemma}\label{lemma:maps}
Let $G$ be a Hausdorff group.
\begin{compactitem}
\item Suppose that $H$ is a subgroup.
	We write $H\rightfactor G$ to denote the set of right cosets $Hg$, $g\in G$,
	and equip $H\rightfactor G$ with the quotient topology.
	Then the projection $\pi\colon G\to H\rightfactor G$ is open
	(i.e.\ images of open sets are open).
	If $H$ is compact then $\pi$ is also proper
	(i.e.\ preimages of compact sets are compact).
\item Suppose that $H$ is a Hausdorff group and
	$\pi\colon H\to G$ is a continuous and open homomorphism which is onto.
	If the kernel of $\pi$ is compact then $\pi$ is proper.
\end{compactitem}
\end{lemma}

\begin{example}\label{example:exact}
Let $G$ be a compactly generated, locally compact Hausdorff group with word metric $d_G$, $N$ a compact Hausdorff group, and $H$ a Hausdorff group. Suppose that
\[ \{1\} \longrightarrow N \longrightarrow H \stackrel{\pi}{\longrightarrow} G \longrightarrow \{1\} \]
is a topological exact sequence (i.e.\ all involved homomorphisms are continuous). The action $H\times G\to G$, $(h,g)\mapsto\pi(h)g$ is continuous and it acts by isometries. As $\pi$ is onto, this action is obviously cobounded. Furthermore, the action is proper, if and only if
\[ \{ h\in H \setsep d_G(hg,g) \le r \} = \pi^{-1}(g B(1,r) g^{-1}) \]
is compact for all $g\in G$ and all $r\ge0$. Here $B(1,r)$ is the closed ball in $G$ with respect to $d_G$. If $\pi$ is an open map, it follows that the action is proper (Lemma~\ref{lemma:maps}) and $H$ is locally compact, since this is an extension property.
\end{example}

\begin{example}\label{example:subgroup}
Consider a compactly generated, locally compact Hausdorff group $G$ with word metric $d_G$ and let $H$ be a subgroup of $G$. Then $H\times G\to G$, $(h,g)\to hg$ is a continuous action which acts by isometries. The set $H\rightfactor G$ inherits a metric $d_{H\rightfactor G}$ from $G$:
\[ d_{H\rightfactor G}(Hg_1, Hg_2) = \min\{d_G(h_1g_1, h_2g_2) \setsep h_1, h_2\in H\} \]
for $g_1,g_2\in G$, which is well-defined, since $d_G$ is discrete. By left-invariance the action is cobounded, if and only if $(H\rightfactor G, d_{H\rightfactor G})$ is bounded. Notice that $(H\rightfactor G, d_{H\rightfactor G})$ is bounded, if $H\rightfactor G$ is compact with respect to the quotient topology of $G$. To see this, choose $n\ge1$, such that $S^n$ contains an open neighborhood $U$ of $1$. Since the projection $\pi\colon G\to H\rightfactor G$ is open (Lemma~\ref{lemma:maps}), $\{\pi(gU) \setsep g\in G\}$ is an open cover of $H\rightfactor G$. Hence there is a finite subcover $\{\pi(g_1U),\dotsc,\pi(g_mU)\}$. Thus any coset of $H\rightfactor G$ is of the form $Hg_iu$ for some $1\le i\le m$ and some $u\in U$. This yields the bound
\begin{align*}
d_{H\rightfactor G}(H,Hg_iu)
&\le d_G(1,g_iu) \le d_G(1,g_i) + d_G(1,u) \\
&\le \max\{d_G(1,g_i) \setsep 1\le i\le m\} + n.
\end{align*}
If $H$ is a closed subgroup then $H$ is locally compact and this action is proper. To see this let $g\in G$ and $r\ge0$ be given. Then
\[ \{ h\in H \setsep d_G(hg,g) \le r \} = g B(1,r) g^{-1} \cap H \]
is compact, since $g B(1,r) g^{-1}$ is compact and $H$ is closed.
\end{example}

By an application of the generalized Milnor-\v Svarc lemma to the situations described in the two previous examples we obtain the following:

\begin{corollary}\label{corollary:milnor1}
Consider a compactly generated, locally compact Hausdorff group $G$ with word metric $d_G$.
\begin{itemize}
\item Suppose that $N$ is a compact Hausdorff group and $H$ is a Hausdorff group and that
	\[ \{1\} \longrightarrow N \longrightarrow H \stackrel{\pi}{\longrightarrow} G \longrightarrow \{1\} \]
	is a topological exact sequence, such that $\pi\colon H\to G$ is open.
	Then $H$ is compactly generated and locally compact and
	$\pi$ is a quasi-isometry from $(H,d_H)$ to $(G,d_G)$
	for any word metric $d_H$ on $H$.
\item If $H$ is a closed subgroup of $G$ and $(H\rightfactor G, d_{H\rightfactor G})$ is bounded
	then $H$ is compactly generated and locally compact and
	the inclusion is a quasi-isometry from $(H,d_H)$ to $(G,d_G)$
	for any word metric $d_H$ on $H$.
	Furthermore, if $H\rightfactor G$ is compact,
	then $(H\rightfactor G, d_{H\rightfactor G})$ is bounded.
\end{itemize}
\end{corollary}

Finally, we note the following consequence of the Milnor-\v Svarc lemma, which says, that any reasonable metric on a compactly generated, locally compact Hausdorff group is quasi-isometrically equivalent to any word metric on the group.

\begin{corollary}\label{corollary:milnor2}
Let $G$ be a locally compact Hausdorff group. Suppose that $d_Q$ is a left-invariant, $q$\ndash quasi-geodesic metric on $G$ with the property, that compact subsets are bounded with respect to $d_Q$ and closed balls with respect to $d_Q$ are compact. Then $G$ is compactly generated and $d_Q$ is quasi-isometrically equivalent to any word metric on $G$.
\end{corollary}

Note that it is not assumed that the metric $d_Q$ induces the group topology. However, the assumptions guarantee some compatibility between the metric $d_Q$ and the group topology. For example, the assumptions on $d_Q$ are satisfied, if $d_Q$ is left-invariant, geodesic, proper and induces the group topology.

\section{Nilpotent Lie groups}
\label{appendix:groups}

The purpose of the appendix is to provide some background on nilpotent Lie groups, see for instance \cite{corwin1990representations,goodman1976nilpotent,hochschild1965structure}, and, mainly, to prove several technical results, which are used in the proof of Theorem~\ref{theorem:nilpotent}.

Let $G$ be a group. We denote by $[g,h] = g^{-1}h^{-1}gh$ the commutator in $G$ and define the $k$\ndash fold commutator inductively by $[g_1]=g_1$ and $[g_1,\dotsc,g_k]=[g_1,[g_2,\dotsc,g_k]]$. The \emph{descending central series} of $G$ is inductively defined by
\[ \gamma_1(G) = G \qquad\text{and}\qquad \gamma_{n+1}(G) = \gen{ [G,\gamma_n(G)] } \]
for $n\ge1$. A group $G$ is called \emph{nilpotent} if $\gamma_{n+1}(G)=\{1\}$ for some integer $n$ and the least integer $n$ with this property is called \emph{nilpotency class} of $G$. If $A$ is a subset of $G$ then the set
\[ I(A) = \{ g\in G \setsep g^n\in A \text{ for some } n\in\N \} \]
is called \emph{isolator} of $A$.

If $G$ is commutative and finitely generated, we denote its torsion-free rank by $\rank(G)$. If $G$ is a commutative, connected Lie group then $G$ is isomorphic to $\R^a \times (\R/\Z)^b$ for some integers $a,b$. In analogy to the discrete case we call $a$ the \emph{compact-free dimension} of $G$ and denote it by $\vdim(G)$.

\begin{lemma}\label{lemma:ranks}
Let $G$ be a nilpotent group and set $G_n=\gamma_n(G)$ for $n\in\N$.
\begin{compactitem}
\item If $G$ is additionally a connected Lie group
	then the set $C$ of all compact elements in $G$
	is a characteristic, connected, compact subgroup,
	$G/C$ is simply connected and
	\[ \vdim(\gamma_n(G/C)/\gamma_{n+1}(G/C)) = \vdim(G_n/G_{n+1}) \]
	for all $n\in\N$.
\item If $G$ is finitely generated
	then the set $T$ of torsion elements in $G$
	is a characteristic, finite subgroup, $G/T$ is torsion-free and
	\[ \rank(\gamma_n(G/T)/\gamma_{n+1}(G/T)) = \rank(G_n/G_{n+1}) \]
	for all $n\in\N$.
\item If $G$ is finitely generated and torsion-free
	then $G = I(G_1) \supseteq I(G_2) \supseteq \dotsb$
	is a central series of $G$ with torsion-free quotients,
	$G_n$ has finite index in $I(G_n)$ and
	\[ \rank(I_n(G)/I_{n+1}(G)) = \rank(G_n/G_{n+1}) \]
	for all $n\in\N$
\end{compactitem}
\end{lemma}

\begin{proof}
Let $G$ be a connected, nilpotent Lie group. Theorem~5.1 in \cite{gluskov1955locally} implies the statements concerning $C$ and $G/C$. It remains to show the equality concerning dimensions. By induction we have $\gamma_n(G/C) = G_nC/C$ and it is easy to check that
\[ G_n/G_{n+1}\to (G_nC/C)/(G_{n+1}C/C), \quad gG_{n+1}\mapsto gC \cdot (G_{n+1}C/C) \]
is a continuous epimorphism with compact kernel which implies the equality.

Now let $G$ be a finitely generated, nilpotent group. Corollary~1.10 in \cite{segal1983polycyclic} yields the first part and the assertion concerning ranks follows mutatis mutandis.

Finally, assume that $G$ is a finitely generated, torsion-free, nilpotent group. By Lemma~3.4 in \cite{segal1983polycyclic} $I(G_1)\supseteq I(G_2)\supseteq \dotsc$ is a central series with torsion-free quotients. Furthermore, it is easy to see that $I(G_n)/G_n = T(G/G_n)$, where $T(G/G_n)$ is the characteristic, finite subgroup of all torsion elements in $G/G_n$. Consider the map
\[ G_n/G_{n+1} \to I(G_n)/I(G_{n+1}), \quad gG_{n+1} \mapsto gI(G_{n+1}). \]
This is a homomorphism which has finite kernel and an image of finite index. This yields the claim concerning ranks.
\end{proof}

In the following we fix a connected, simply connected, nilpotent Lie group $G$ with nilpotency class $c$ and set $G_n=\gamma_n(G)$ for $n\in\N$. We denote by $\LieG$ the associated Lie algebra and by $(x,y)$ the Lie bracket of $\LieG$. Furthermore, we define the $k$\ndash fold Lie bracket inductively by $(x_1)=x_1$ and $(x_1,\dotsc,x_k)=(x_1,(x_2,\dotsc,x_k))$. The \emph{descending central series} of $\LieG$ is
\[ \LieG_1 = \LieG \qquad\text{and}\qquad \LieG_{n+1} = \vspan_\R(\LieG,\LieG_n) \]
for $n\ge1$. The Lie algebra of $G_n$ is $\LieG_n$. Let $\nu(n)$ be the compact-free dimension of $G_n/G_{n+1}$. Then
\[ G_n/G_{n+1} \simeq \LieG_n/\LieG_{n+1} \simeq \R^{\nu(n)} \]
as commutative groups. The \emph{exponential map} $\exp\colon\LieG\to G$ is a diffeomorphism from $\LieG$ to $G$ and its inverse is $\log\colon G\to\LieG$. The Baker-Campbell-Hausdorff formula yields a multiplicative group structure on $\LieG$:
\[ xy = x+y + \tfrac12(x,y) + \tfrac1{12}(x,x,y) - \tfrac1{12}(y,x,y) - \tfrac1{24}(y,x,x,y) \pm \dotsb \]
for $x,y\in\LieG$. Then the exponential map $\exp$ is a group isomorphism from $(\LieG,\cdot)$ to $(G,\cdot)$ and it is common to identify the Lie group $G$ with its Lie algebra $\LieG$.

A subgroup $\Gamma$ is called \emph{uniform} in $G$, if $\Gamma$ is discrete and the quotient $\Gamma\rightfactor G$ is compact. In the following lemma we study uniform subgroups. Its proof depends on well-known results on such subgroups which can be found in \cite[Chapter~5]{corwin1990representations}.

\begin{lemma}\label{lemma:maltsev}
Let $\Gamma$ be a uniform subgroup in $G$ and set $\Gamma_n=\gamma_n(\Gamma)$ for $n\in\N$. Then $\Gamma\cap G_n = I(\Gamma_n)$ and
\[ \rank(\Gamma_n/\Gamma_{n+1}) = \vdim(G_n/G_{n+1}) \]
for all $n\in\N$
\end{lemma}

\begin{proof}
First we show that $\Gamma\cap\gamma_n(G) = I(\gamma_n(\Gamma))$ for all $n\in\N$ by backward induction on $n$:
\begin{itemize}
\item Suppose that $n=c$: Obviously, $I(\Gamma_c)\subseteq\Gamma$ and
	$I(\Gamma_c)\subseteq G_c$, hence $I(\Gamma_c)\subseteq G_c\cap\Gamma$.
	To prove the reversed inclusion, note that
	$\exp$ is a group homomorphism from $(\LieG_c,+)$ to $(G_c,\cdot)$.
	Let $X\subseteq\LieG$ be a strong Mal'tsev basis
	strongly based on $\Gamma$ and set $Z=\exp(X)$.
	Then $\Gamma_c=\gen{[Z,\dotsc,Z]}$
	(see \cite[Theorem~5.4]{magnus2004combinatorial}) and
	thus $\log(\Gamma_c)=\vspan_\Z (X,\dotsc,X)$,
	since $\exp((x_1,\dotsc,x_c))=[\exp(x_1),\dotsc,\exp(x_c)]$
	for all $x_1,\dotsc,x_c\in\LieG$.
	Furthermore, we have $\LieG_c=\vspan_\R (X,\dotsc,X)$.
	This implies that $\Gamma_c$ and $G_c\cap\Gamma$
	are uniform subgroups in $G_c$.
	Therefore $(G_c\cap\Gamma)/\Gamma_c$ is finite,
	whence $G_c\cap\Gamma\subseteq I(\Gamma_c)$.
\item Assume that the claim holds for $n\ge2$:
	Consider the groups $G/G_n$ and $\Gamma G_n/G_n$.
	Then $\Gamma G_n/G_n$ is (topologically) isomorphic
	to $\Gamma/(\Gamma\cap G_n)$.
	By $\phi$ we denote the canonical isomorphism
	$\Gamma G_n/G_n\to\Gamma/(\Gamma\cap G_n)$.
	Since $\Gamma G_n/G_n$ is a uniform subgroup in $G/G_n$ and
	$G/G_n$ is nilpotent with nilpotency class $n-1$,
	using the initial step for the nilpotent group $G/G_n$ yields
	\[ (\Gamma\cap G_{n-1})G_n/G_n
		= \Gamma G_n/G_n\cap\gamma_{n-1}(G/G_n)
		= I(\gamma_{n-1}(\Gamma G_n/G_n)). \]
	Applying the isomorphism $\phi$ on both sides we obtain
	\begin{align*}
	(\Gamma\cap G_{n-1})/(\Gamma\cap G_n)
	&= I(\gamma_{n-1}(\Gamma/(\Gamma\cap G_n))) \\
	&= I(\Gamma_{n-1}(\Gamma\cap G_n))/(\Gamma\cap G_n) \\
	&= I(\Gamma_{n-1})/(\Gamma\cap G_n)
	\end{align*}
	using the induction hypothesis $\Gamma\cap G_n=I(\Gamma_n)$ once more.
	It follows that $\Gamma\cap G_{n-1}=I(\Gamma_{n-1})$.
\end{itemize}
Now we prove the assertion concerning ranks. Since $\Gamma\cap G_k$ is uniform in $G_k$ for all $k\ge1$, it follows that $(\Gamma\cap G_n)G_{n+1}/G_{n+1}$ is uniform in $G_n/G_{n+1}$. This implies that
\[ \rank((\Gamma\cap G_n)/(\Gamma\cap G_{n+1})) = \rank((\Gamma\cap G_n)G_{n+1}/G_{n+1}) = \vdim(G_n/G_{n+1}) \]
which yields the statement using the last part of Lemma~\ref{lemma:ranks}.
\end{proof}

Since $\LieG$ is a real vector space of finite dimension $\nu(1)+\dotsb+\nu(c)$, there are linear subspaces $V_n\subseteq\LieG$ of dimension $\nu(n)$, such that $\LieG_n = V_n \oplus \LieG_{n+1}$. Hence
\[ \LieG_n = V_n \oplus \dotsb \oplus V_c. \]
Write $\pi_n\colon\LieG\to V_n$ to denote the canonical projection. Then $\pi_n$ is a continuous epimorphism from $(\LieG_n,\cdot)$ to $(V_n,+)$ with kernel $\LieG_{n+1}$. Let $\norm{\argument}_n$ be some $\ell^2$\ndash norm on $V_n$. Then
\[ \norm{x} = \max\{ \norm{\pi_n(x)}_n \setsep 1\le n\le c \} \]
is a norm on $\LieG$. Notice that $\norm{\pi_n(x)} = \norm{\pi_n(x)}_n$. Since the Lie bracket $(\argument,\argument)$ is bilinear, we have the following simple statement.

\begin{lemma}\label{lemma:liebracket}
There is a constant $M\ge1$, such that $\norm{(x,y)}\le M\norm{x}\,\norm{y}$ for all $x,y\in\LieG$. Consequently,
\[ \norm{(x_1,\dotsc,x_k)} \le M^{k-1} \norm{x_1}\dotsm \norm{x_k} \]
for all $x_1,\dotsc,x_k\in\LieG$.
\end{lemma}

For $x\in\LieG$ set
\[ \gauge{x} = \max\{ \norm{\pi_n(x)}^{1/n} \setsep 1\le n\le c \}. \]
Then $\gauge{\argument}$ is called (homogeneous) \emph{gauge} or \emph{quasi-norm} (see for instance \cite{breuillard2012geometry,goodman1976nilpotent,guivarch1973croissance}). Note that $\gauge{\argument}$ is homogeneous with respect to the dilation $\delta_t(x) = t\pi_1(x) + \dotsb + t^c\pi_c(x)$, i.e.\ $\gauge{\delta_t(x)}=t\gauge{x}$, and it satisfies a weak form of the triangle inequality with respect to the Lie group structure on $\LieG$ (see Lemma~\ref{lemma:gaugenorm}).

\begin{lemma}\label{lemma:gaugeprop}
For all $x,y\in\LieG$ the following holds:
\begin{compactitem}
\item $\gauge{-x}=\gauge{x}$,
\item $\gauge{x+y}\le \gauge{x}+\gauge{y}$,
\item if $x\in\LieG_n$ and $\alpha\ge1$ then $\gauge{\alpha x}\le\alpha^{1/n}\gauge{x}$,
\item if $0\le\alpha\le1$ then $\gauge{\alpha x}\le\alpha^{1/c}\gauge{x}$.
\end{compactitem}
In any case, $\gauge{\alpha x} \le \max\{1,\alpha\} \gauge{x}$ for all $\alpha\ge0$.
\end{lemma}

The following lemma is a crucial observation due to Guivarc'h~\cite[Lemme~II.1]{guivarch1973croissance}, see also \cite[Lemma~2.5]{breuillard2012geometry}.

\begin{lemma}\label{lemma:gaugenorm}
Let $\alpha>0$. Then, by appropriately rescaling the norms $\norm{\argument}_n$, we have
\[ \gauge{xy} \le \gauge{x} + \gauge{y} + \alpha \]
for all $x,y\in\LieG$.
\end{lemma}

In the sequel we assume that the norms $\norm{\argument}_n$ are chosen appropriately, so that the previous lemma holds with $\alpha=1$. As a simple consequence we obtain $\gauge{(x,y)}\le 2\gauge{x}+2\gauge{y}+2$ and it follows by induction, that
\begin{equation}\label{equation:liebracket}
\gauge{(x_1,\dotsc,x_k)} \le 2^{k-1} (\gauge{x_1}+\dotsb+\gauge{x_k}) + 2^k
\end{equation}
for all $x_1,\dotsc,x_k\in\LieG$.

Since $(G,\cdot)\simeq (\LieG,\cdot)$ is a connected, locally compact group, it is compactly generated. Let $d_w$ be some word metric on the group $(\LieG,\cdot)$. The following result shows a fundamental connection between the gauge $\gauge{\argument}$ and the word metric $d_w$.

\begin{lemma}[Theorem~2.7 in \cite{breuillard2012geometry}]\label{lemma:gaugequasi}
There is a constant $q\ge1$, such that
\[ q^{-1} \gauge{x} \le d_w(0,x) \le q \gauge{x} + q \]
for all $x\in\LieG$.
\end{lemma}

After providing the basic setup and important tools from Lie theory, we now apply the notions of Section~\ref{section:construction} to this setting. We write $s^+_w$ instead of $s^+_{\smash{(\LieG,d_w)}}$. The quantity $d_a$ defined by $d_a(x,y) = \gauge{-x+y}$ yields by Lemma~\ref{lemma:gaugeprop} a metric on $\LieG$, and as before we write $s^+_a$ instead of $s^+_{\smash{(\LieG,d_a)}}$. Although $(x,y)\mapsto \gauge{x^{-1}y}$ is not a metric, we define
\[ s^+_m(\genp x, \genp y)
	= \limsup_{n\to\infty} \, \inf\biggl\{ \frac{\gauge{y^{-n}x^m}}{\gauge{x^m}} \setsep m\in\N_0 \biggr\} \]
and
\[ s^+_m(\gen x, \gen y)
	= \limsup_{\abs{n}\to\infty} \, \inf\biggl\{ \frac{\gauge{y^{-n}x^m}}{\gauge{x^m}} \setsep m\in\N_0 \biggr\} \]
for $x,y\in\LieG\setminus\{0\}$. Using Lemma~\ref{lemma:limsup} and Lemma~\ref{lemma:gaugequasi} we get the following comparison of $s^+_w$ and $s^+_m$.

\begin{lemma}\label{lemma:compare}
Let $x,y\in\LieG$ with $x\ne0$ and $y\ne0$. Then
\[ q^{-2} s^+_m(\genp x,\genp y) \le s^+_w(\genp x,\genp y) \le q^2 s^+_m(\genp x,\genp y) \]
and
\[ q^{-2} s^+_m(\gen x,\gen y) \le s^+_w(\gen x,\gen y) \le q^2 s^+_m(\gen x,\gen y), \]
where $q$ is the constant of Lemma~\ref{lemma:gaugequasi}.
\end{lemma}

Our goal is the comparison of $s^+_a$ and $s^+_m$. We restrict this comparison to elements of $\Forward\LieG$ and $\Orbit\LieG$. Note that $\Forward(\LieG,\cdot)=\Forward(\LieG,+)$ and $\Orbit(\LieG,\cdot)=\Orbit(\LieG,+)$, since $x^n=nx$ for all $x\in\LieG$ and $n\in\Z$. Before we provide the necessary tools for this comparison, let us identify $\Linear(\LieG,d_a)$ and $\Project(\LieG,d_a)$.

\begin{lemma}\label{lemma:additive}
Up to homeomorphism we have
\[ \Linear(\LieG,d_a) = \S^{\nu(1)-1} \uplus\dotsb\uplus \S^{\nu(c)-1}, \qquad
	\Project(\LieG,d_a) = \P^{\nu(1)-1} \uplus\dotsb\uplus \P^{\nu(c)-1}. \]
Moreover, the following three statements yield a precise description of $\Linear(\LieG,d_a)$ and $\Project(\LieG,d_a)$.
\begin{compactenum}
\def\theenumi{(\alph{enumi})}
\def\labelenumi{\normalfont\theenumi}
\item\label{claim:firsta}
	If $x,y\in\LieG_i$ and $x+\LieG_{i+1}=y+\LieG_{i+1}\ne\LieG_{i+1}$ then
	\[ s^+_a(\genp x,\genp y)=0 \qquad\text{and}\qquad s^+_a(\gen x,\gen y)=0. \]
\item\label{claim:seconda}
	If $x\in\LieG_i$, $x\notin\LieG_{i+1}$, and $y\in\LieG_{i+1}$ then
	\[ s^+_a(\genp x,\genp y)=1 \qquad\text{and}\qquad s^+_a(\gen x,\gen y)=1. \]
\item\label{claim:thirda}
	If $x,y\in V_i$ and $x,y\ne0$ then, using the notation of Example~\ref{example:euclid},
	\[ s^+_a(\genp x,\genp y) = \bigl(\sin(\min\{\tfrac12\pi,\angle(H_x,H_y)\})\bigr)^{1/i} \]
	and
	\[ s^+_a(\gen x,\gen y) = \bigl(\sin(\angle(L_x,L_y))\bigr)^{1/i}. \]
\end{compactenum}
\end{lemma}

\begin{proof}
Once we have proved \ref{claim:firsta}, \ref{claim:seconda}, \ref{claim:thirda} the statement of the lemma follows. We only prove these three statements for $s^+_a(\genp x,\genp y)$ the other case being analogous.

\textit{Statement}~\ref{claim:firsta}. By assumption $-y+x\in\LieG_{i+1}$, whence
	\[ \gauge{-ny+nx} = \gauge{n(-y+x)} \le n^{1/(i+1)}\gauge{-y+x}. \]
	Since $x\in\LieG_i\setminus\LieG_{i+1}$, it follows that $\pi_i(x)\ne0$ and
	\[ \gauge{nx} \ge \gauge{\pi_i(nx)} = n^{1/i}\gauge{\pi_i(x)}. \]
	From this we infer that
	\[ s^+_a(\genp x, \genp y) \le \limsup_{n\to\infty} \frac{\gauge{-ny+nx}}{\gauge{nx}}
		\le \limsup_{n\to\infty} \frac{n^{1/(i+1)}\gauge{-y+x}}{n^{1/i}\gauge{\pi_i(x)}} = 0. \]

\textit{Statement}~\ref{claim:seconda}. Using \ref{claim:firsta},
	we may assume that $x\in V_i$. Then $\pi_i(-ny+mx)=mx$ and so
	\[ \gauge{-ny+mx} \ge \gauge{\pi_i(-ny+mx)} = \gauge{mx}. \]
	This implies
	\[ \inf\biggl\{ \frac{\gauge{-ny+mx}}{\gauge{mx}} \setsep m\in\N_0 \biggr\} \ge 1 \]
	and therefore $s^+_a(\genp x, \genp y) \ge 1$.

\textit{Statement}~\ref{claim:thirda}. Note that $\gauge{v}=\norm{v}^{1/i}$ for all $v\in V_i$.
	Since $s^+_a(\genp x, \genp y) = s^+_a(H_x,H_y)$,
	the statement follows from Example~\ref{example:euclid}.
\end{proof}

We now compare $s^+_a$ and $s^+_w$. Let $y,z$ be elements in $\LieG$ and consider the product $y^{-1}(y+z) = (-y)(y+z)$. Then, using the Baker-Campbell-Hausdorff formula,
\begin{equation}\label{equation:product}
\begin{aligned}
y^{-1}(y+z)
&= (-y) + (y+z) + \tfrac12(-y,y+z) + \tfrac1{12}(-y,-y,y+z) \\
&\qquad\qquad - \tfrac1{12}(y+z,-y,y+z) \pm \dotsb \\
& = z - \tfrac12(y,z) + \tfrac2{12}(y,y,z) + \tfrac{1}{12}(z,y,z) \pm \dotsb.
\end{aligned}
\end{equation}
Of course in the last expression above at most $c$\ndash fold Lie brackets occur and, for each $1\le k\le c$, there are finitely many $k$\ndash fold Lie brackets, say $v_{k,1},\dotsc,v_{k,m(k)}$, whose entries are either $y$ or $z$, and each of which contains at least one $y$ and at least one $z$. If $1\le k\le c$ and $1\le j\le m(k)$ then write $q_{k,j}$ for the rational coefficient in front of the $k$\ndash fold Lie bracket $v_{k,j}$. Then
\[ y^{-1}(y+z) = \sum_{1\le k\le c} \sum_{1\le j\le m(k)} q_{k,j} v_{k,j}. \]
Note that the constants $q_{k,j}$ depend on the Baker-Campbell-Hausdorff formula only. For convenience we set $Q_{k,j} = \max\{1,q_{k,j}\}$ and
\[ Q = \sum_{1\le k\le c} \sum_{1\le j\le m(k)} Q_{k,j}. \]

\begin{lemma}\label{lemma:equal}
Suppose that $x,y\in\LieG_i$ and $x\LieG_{i+1}=y\LieG_{i+1}\ne\LieG_{i+1}$. Then
\[ \gauge{y^{-n}x^n} \le 2^{c-1}Q (c\gauge{x}+c\gauge{y}+2) n^{(1-1/c)/i}  \]
for all $n\ge0$.
\end{lemma}

\begin{proof}
Set $z=x-y$ and $m=\gauge{x}+\gauge{y}$. By assumption $z\in\LieG_{i+1}$ and obviously $\gauge{x},\gauge{y},\gauge{z}\le m$. Using the representation~\eqref{equation:product} of the product $y^{-1}(y+z)$ we obtain
\[ y^{-n} x^n = y^{-n} (y+z)^n = \sum_{1\le k\le c} \sum_{1\le j\le m(k)} q_{k,j} n^k v_{k,j}. \]
Since each $k$\ndash fold Lie bracket $v_{k,j}$ contains at least one $z$, we get $v_{k,j}\in\LieG_{ki+1}$. Using \eqref{equation:liebracket} yields $\gauge{v_{k,j}}\le 2^{k-1} km + 2^k = 2^{k-1}(km+2)$ for all $k,j$ and therefore
\[ \gauge{q_{k,j} n^k v_{k,j}} \le Q_{k,j} n^{k/(ki+1)} 2^{k-1} (km+2). \]
Collecting the pieces, we obtain
\begin{align*}
\gauge{y^{-n} x^n}
&\le \sum_{1\le k\le c} \sum_{1\le j\le m(k)} \gauge{q_{k,j} n^k v_{k,j}} \\
&\le \sum_{1\le k\le c} \sum_{1\le j\le m(k)} Q_{k,j} n^{k/(ki+1)} 2^{k-1} (km+2) \\
&\le 2^{c-1}Q (cm+2) n^{(1-1/c)/i} \qedhere
\end{align*}
\end{proof}

\begin{lemma}\label{lemma:upper}
Suppose that $x,y\in V_i$ and $\gauge{x}\ge\gauge{y}=1$ and $\gauge{x-y}=\alpha\gauge{x}$ for some $\alpha\in[0,1]$. Then
\[ \gauge{y^{-n}x^n} \le MQ \alpha^{i/c} \gauge{x^n} \]
for all $n\ge0$.
\end{lemma}

\begin{proof}
Set $z=x-y\in V_i$. Of course $\norm{x}\ge\norm{y}=1$, and $\norm{z}=\alpha^i\norm{x}$. Using the representation~\eqref{equation:product} we get as in the proof above
\[ y^{-n} x^n = y^{-n} (y+z)^n = \sum_{1\le k\le c} \sum_{1\le j\le m(k)} q_{k,j} n^k v_{k,j}. \]
Each $k$\ndash fold Lie bracket $v_{k,j}$ contains at least one $z$, but this time $v_{k,j}\in\LieG_{ki}$. An application of Lemma~\ref{lemma:liebracket} implies
\begin{align*}
\gauge{v_{k,j}}
&= \max\{ \norm{\pi_j(v_{k,j})}^{1/l} \setsep ik\le l\le c \} \\
&\le \max\{ \norm{v_{k,j}}^{1/l} \setsep ik\le l\le c \} \\
&\le \max\{ (M^{k-1} \alpha^i \norm{x}^k)^{1/l} \setsep ik\le l\le c \} \\
&\le M \alpha^{i/c} \norm{x}^{1/i} = M \alpha^{i/c} \gauge{x}.
\end{align*}
Hence we obtain
\begin{align*}
\gauge{y^{-n} x^n}
&\le \sum_{1\le k\le c} \sum_{1\le j\le m(k)} \gauge{q_{k,j} n^k v_{k,j}} \\
&\le \sum_{1\le k\le c} \sum_{1\le j\le m(k)} Q_{k,j} n^{1/i} M \alpha^{i/c} \gauge{x} \\
&= MQ \alpha^{i/c} \gauge{x^n} \qedhere
\end{align*}
\end{proof}

\begin{lemma}\label{lemma:three}
The following three statements hold.
\begin{compactenum}
\def\theenumi{(\alph{enumi})}
\def\labelenumi{\normalfont\theenumi}
\item\label{claim:firstm}
	If $x,y\in\LieG_i$ and $x\LieG_{i+1}=y\LieG_{i+1}\ne\LieG_{i+1}$ then
	\[ s^+_m(\genp x,\genp y)=0 \qquad\text{and}\qquad s^+_m(\gen x,\gen y)=0. \]
\item\label{claim:secondm}
	If $x\in\LieG_i$, $x\notin\LieG_{i+1}$, and $y\in\LieG_{i+1}$ then
	\[ s^+_m(\genp x,\genp y)=1 \qquad\text{and}\qquad s^+_m(\gen x,\gen y)=1. \]
\item\label{claim:thirdm}
	If $x,y\in V_i$ and $x,y\ne0$ then
	\[ s^+_a(\genp x,\genp y) \le s^+_m(\genp x,\genp y) \le MQ \bigl(s^+_a(\genp x,\genp y)\bigr)^{i/c} \]
	and
	\[ s^+_a(\gen x,\gen y) \le s^+_m(\gen x,\gen y) \le MQ \bigl(s^+_a(\gen x,\gen y)\bigr)^{i/c}. \]
\end{compactenum}
\end{lemma}

\begin{proof}
\textit{Statement}~\ref{claim:firstm}. By assumption $\pi_i(x)\ne0$ and we get
	\[ \gauge{x^n}\ge\gauge{\pi_i(x^n)}=n^{1/i}\gauge{\pi_i(x)}. \]
	On the other hand Lemma~\ref{lemma:equal} implies
	\[ \gauge{y^{-n}x^n} \le 2^{c-1}Q(c\gauge{x}+c\gauge{y}+2) n^{(1-1/c)/i} \]
	for all $n\ge0$. Hence
	\begin{align*}
	s^+_m(\genp x,\genp y)
	&\le \limsup_{n\to\infty} \frac{\gauge{y^{-n}x^n}}{\gauge{x^n}} \\
	&\le \limsup_{n\to\infty} \frac{2^{c-1}Q(c\gauge{x}+c\gauge{y}+2) n^{(1-1/c)/i}}{n^{1/i}\gauge{\pi_i(x)}} = 0.
	\end{align*}

\textit{Statement}~\ref{claim:secondm}: By the first claim we may assume that $x\in V_i$.
	Using the Baker-Campbell-Hausdorff formula we obtain $\pi_i(y^{-n}x^m) = x^m$ and thus
	\[ \gauge{y^{-n}x^m} \ge \gauge{\pi_i(y^{-n}x^m)} = \gauge{x^m}. \]
	This implies
	\[ \inf\biggl\{ \frac{\gauge{y^{-n}x^m}}{\gauge{x^m}} \setsep m\in\N_0 \biggl\} \ge 1 \]
	and $s^+_m(\genp x, \genp y)\ge1$.

\textit{Statement}~\ref{claim:thirdm}:
	To prove the lower bound, note that
	\[ \gauge{y^{-n}x^m} \ge \gauge{\pi_i(y^{-n}x^m)} = \gauge{-ny+mx} \]
	for all $n,m\in\N_0$.
	This implies $s^+_m(\genp x,\genp y) \ge s^+_a(\genp x,\genp y)$.

	Now we prove the upper bound. Set $\alpha=s^+_a(\genp x,\genp y)$.
	Without loss of generality we may assume that $\alpha<1$.
	Furthermore, we may scale $x$ and $y$ by positive constants
	without changing the value of $s^+_a(\genp x,\genp y)$ or of $s^+_m(\genp x,\genp y)$.
	Hence we may assume that $\norm{y}=1$ and $y$ is orthogonal to $x-y$
	with respect to the inner product on $V_i$ associated with $\norm{\argument}$,
	see Figure~\ref{figure:choice}.
	\begin{figure}[htb]
	\centering
	\tikzstyle{vertex}=[circle, fill=black, inner sep=0pt, minimum width=2pt]
	\begin{tikzpicture}[scale=0.6]
		\draw (0,0)--(6,0);
		\draw (0,0)--(6,3);
		\draw (3,1.5) node[vertex] {} node[above] {$x$} -- (3,0) node[vertex] {} node[below] {$y$};
		\draw (3,0.4) arc(90:180:0.4);
		\draw (3+1.5,1.5) arc(0:360:1.5);
		\draw (1.5,0) node[below] {$1$};
	\end{tikzpicture}
	\caption{The constraints for the choice of $x$ and $y$.}
	\label{figure:choice}
	\end{figure}
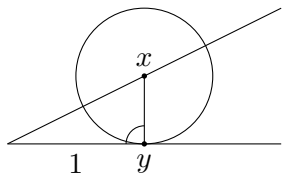
	As a consequence we get $1=\norm{y}\le\norm{x}$ and $\norm{x-y} = \alpha^i \norm{x}$
	(due to Lemma~\ref{lemma:additive}).
	Then $1=\gauge{y}\le\gauge{x}$ and $\gauge{x-y}=\alpha\gauge{x}$.
	By Lemma~\ref{lemma:upper} we get
	\[ \gauge{y^{-n}x^n} \le MQ \alpha^{i/c} \gauge{x^n} \]
	for all $n\ge0$. Thus
	\[ s^+_m(\genp x,\genp y)
		\le \limsup_{n\to\infty} \frac{\gauge{y^{-n}x^n}}{\gauge{x^n}}
		\le MQ\alpha^{i/c}. \qedhere \]
\end{proof}

\end{appendix}

\def\doi#1{\href{http://dx.doi.org/#1}{\protect\nolinkurl{doi:#1}}}
\bibliographystyle{amsalphaurl}
\bibliography{bndry}

\parindent=0pt

\end{document}